\documentclass[matprg]{svjour3}

\usepackage{amsmath,hyperref}
\usepackage{lipsum}
\usepackage{amsfonts}
\usepackage{graphicx}
\usepackage{epstopdf}
\usepackage{algorithm,algorithmic}
\usepackage{verbatim}
\usepackage{mathtools}

\ifpdf
  \DeclareGraphicsExtensions{.eps,.pdf,.png,.jpg}
\else
  \DeclareGraphicsExtensions{.eps}
\fi

\newcommand{\TheTitle}{Communication-Efficient Algorithms for Decentralized and Stochastic Optimization}

\title{{\TheTitle}\thanks{
This work was funded by National Science Foundation grants 1637473 and 1637474, and Office of Naval Research grant N00014-16-1-2802}}

\author{
  Guanghui Lan\and Soomin Lee\and
  Yi Zhou\thanks{Department of Industrial and Systems Engineering, Georgia Institute of Technology, Atlanta, GA, 30332. (\email{george.lan@isye.gatech.edu, soomin.lee@isye.gatech.edu, yizhou@gatech.edu})}
}

\usepackage{amsopn}

\def\a{\alpha}
\def\b{\beta}

\def\t{\theta}
\def\o{\omega}

\def\Gc{\mathcal{G}}
\def\Vc{\mathcal{N}}  
\def\vb{\mathbf{s}}   

\def\Ec{\mathcal{E}}

\def\1b{\mathbf{1}}
\def\0b{\mathbf{0}}

\def\la{\langle}
\def\ra{\rangle}

\def\Lb{\mathbf{L}}

\def\zb{\mathbf{z}}
\def\xb{\mathbf{x}}
\def\yb{\mathbf{y}}

\def\ub{\mathbf{u}}

\DeclareMathOperator*{\argmin}{arg\,min}

\allowdisplaybreaks
\newcommand\myeqa{\stackrel{\mathclap{\normalfont\mbox{(a)}}}{\le}}
\newcommand\myeqb{\stackrel{\mathclap{\normalfont\mbox{(b)}}}{\le}}
\newcommand\myeqc{\stackrel{\mathclap{\normalfont\mbox{(c)}}}{\le}}
\newcommand\myeqd{\stackrel{\mathclap{\normalfont\mbox{(d)}}}{\le}}
\newcommand\myeqe{\stackrel{\mathclap{\normalfont\mbox{(e)}}}{\le}}

\usepackage{amsfonts}
\usepackage{latexsym}
\usepackage{amsmath, amssymb}
\usepackage{color}

\newtheorem{assumption}{Assumption}

\newcommand{\tsum}{\textstyle{\sum}}
\newcommand{\beq}{\begin{equation}}
\newcommand{\eeq}{\end{equation}}
\newcommand{\nn}{\nonumber}
\newcommand{\bbe}{\mathbb{E}}
\newcommand{\bbr}{\mathbb{R}}

\def\SGD{{\rm CS}}

\def\eqnok#1{\eqref{#1}}
\def\argmin{{\rm argmin}}

\def\exp{{\rm exp}}
\def\vgap{\vspace*{.1in}}
\def\prob{\mathop{\rm Prob}}
\def\Prob{{\hbox{\rm Prob}}}

\begin{document}

\maketitle

\begin{abstract}
We present a new class of decentralized first-order methods for nonsmooth and stochastic optimization problems defined over multiagent networks.
Considering that communication is a major bottleneck in decentralized optimization,
our main goal in this paper is to develop algorithmic frameworks which can significantly reduce the number of inter-node communications. We first propose a decentralized primal-dual method which can find an $\epsilon$-solution
both in terms of functional optimality gap and feasibility residual in ${\cal O}(1/\epsilon)$ inter-node communication
rounds when the objective functions are convex  and the local primal subproblems are solved exactly.
Our major contribution is to present a new class of decentralized primal-dual type algorithms, namely the decentralized communication sliding (DCS) methods,
which can skip the inter-node communications while agents solve the primal subproblems iteratively through linearizations of their local objective functions.
By employing DCS, agents can still find an $\epsilon$-solution in ${\cal O}(1/\epsilon)$ (resp., ${\cal O}(1/\sqrt{\epsilon})$)
communication rounds for general convex functions (resp., strongly convex functions), while maintaining
the ${\cal O}(1/\epsilon^2)$ (resp., ${\cal O}(1/\epsilon)$) bound on the total number of intra-node subgradient evaluations.
We also present a stochastic counterpart for these algorithms, denoted by SDCS,
for solving stochastic optimization problems whose objective function cannot be evaluated exactly.
In comparison with existing results for decentralized nonsmooth and stochastic optimization, we can reduce the total number of inter-node communication rounds by orders of magnitude while still maintaining the optimal complexity bounds on intra-node stochastic subgradient evaluations. The bounds on the
(stochastic) subgradient evaluations are actually comparable to those required for centralized nonsmooth and stochastic optimization
under certain conditions on the target accuracy.

\vspace{.1in}

\noindent {\bf Keywords:}
decentralized optimization, decentralized machine learning, communication efficient, stochastic programming, nonsmooth functions,
primal-dual method, complexity

\vspace{.07in}

\noindent {\bf AMS 2000 subject classification:}  90C25, 90C06, 90C22, 49M37, 93A14, 90C15
\end{abstract}
\setcounter{equation}{0}
\section{Introduction}
Decentralized optimization problems defined over complex multiagent
networks are ubiquitous in
signal processing, machine learning, control, and other areas in science and engineering
(see e.g. \cite{rabbat,con01,ram_info,Durham-Bullo}).
In this paper, we consider the following decentralized optimization problem
which is cooperatively solved by the network of $m$ agents:
\begin{align}\label{eqn:orgprob}
\min_x &~f(x) := \tsum_{i=1}^m f_i(x)\\
\text{s.t. }&~ x \in X, \quad X := \cap_{i=1}^m X_i \nonumber,
\end{align}
where $f_i:X_i \to \mathbb{R}$
is a convex and possibly nonsmooth objective function of agent $i$
satisfying
\begin{align}\label{eqn:nonsmooth}
\tfrac{\mu}{2}\|x-y\|^2\le f_i(x)-f_i(y)-\langle f_i'(y), x-y\rangle \le M\|x-y\|, \ \ \forall x, y\in X_i,
\end{align}
for some $M,\mu\ge 0$ and $f_i'(y) \in \partial f_i(y)$,
where $\partial f_i(y)$ denotes the subdifferential of $f_i$ at $y$,
and $X_i \subseteq \mathbb{R}^d$ is a closed convex constraint set of agent $i$.
Note that $f_i$ and $X_i$ are private and only known to agent $i$.
Throughout the paper, we assume the feasible set $X$ is nonempty.

In this paper, we also consider the situation where one can only have access to
noisy first-order information (function values and subgradients) of the functions $f_i$, $i = 1, \ldots, m$ (see \cite{NJLS09-1,Lan10-3}).
This happens, for example, when the function $f_i$'s are given in the form of expectation, i.e.,
\begin{align} \label{exp_prob}
f_i(x) := \mathbb{E}_{\xi_i}[F_i(x;\xi_i)],
\end{align}
where the random variable $\xi_i$ models a source of uncertainty and the distribution $\mathbb{P}(\xi_i)$
is not known in advance. As a special case of \eqnok{exp_prob}, $f_i$ may be given as the summation of many components, i.e.,
\begin{align} \label{emp_prob}
f_i(x) := \tsum_{j=1}^l f_i^j(x),
\end{align}
where $l \ge 1$ is some large number.
Stochastic optimization problem of this type has great potential of applications in data analysis, especially in machine learning.
In particular, problem \eqnok{exp_prob} corresponds to the minimization of generalized
risk and is particularly useful for dealing with online (streaming) data distributed over a network, while problem \eqnok{emp_prob} aims
at the collaborative minimization of empirical risk. Currently the dominant approach is to
collect all agents' private data on a server (or cluster) and to apply centralized
machine learning techniques. However, this centralization scheme would require agents to submit their private data to
the service provider without much control on how the data will be used, in addition to incurring high setup cost related
to the transmission of data to the service provider. Decentralized optimization
provides a viable approach to deal with these data privacy related issues.

In these decentralized and stochastic optimization problems, each network agent $i$ is associated with the local
objective function $f_i(x)$ and all agents
intend to cooperatively minimize the system objective $f(x)$ as the
sum of all local objective $f_i$'s
in the absence of full knowledge about the global problem and network structure.
A necessary feature in decentralized optimization is, therefore, that the agents
must communicate with their neighboring agents
to propagate the distributed information to every location in the network.

One of the most well-studied techniques in decentralized optimization are the
subgradient based methods (see e.g., \cite{Nedic2009,Nedic11,tsianos-pushsum,ANAO,Duchi-DDA,Lobel2011,WYin-Extra}),
where at each step a local subgradient is taken at each node, followed by the communication
with neighboring agents.
Although the subgradient computation at each step can be
inexpensive, these methods usually require lots of iterations until convergence.
Considering that one iteration in decentralized optimization is equivalent to one communication round among agents,
this can incur a significant latency.
CPUs in these days can read and write the memory at over $10$ GB per
second whereas communication over TCP/IP is about $10$ MB
per second. Therefore, the gap between intra-node computation and inter-node communication is
about 3 orders of magnitude. The communication start-up cost itself
is also not negligible as it usually takes a few milliseconds.

Another well-known type of decentralized algorithm relies on dual methods (see e.g., \cite{Boyd-ADMM,Wei-admm,perturbedPD}),
where at each step for a fixed dual variable, the primal variables
are solved to minimize some local Lagrangian related function,
then the dual variables associated with the consistency constraints
are updated accordingly.
Although these dual type methods usually require fewer numbers of iterations
(hence, fewer communication rounds) than the subgradient methods until convergence,
one crucial problem of these methods is that
the local subproblem associated with each agent cannot be solved efficiently in many cases.

The main goal of this paper is, therefore, to
develop dual based decentralized algorithms for solving \eqref{eqn:orgprob}
that is communication efficient and has local subproblems
easily solved by each agent through the utilization of (noisy) first-order information of $f_i$.
More specifically, we will provide a theoretical understanding
on how many numbers of inter-node communications and intra-node (stochastic) subgradient evaluations of $f_i$
are required in order to find a certain approximate solution of \eqref{eqn:orgprob}.


\subsection{Notation and Terminologies}
Let $\mathbb{R}$ denote the set of real numbers.
All vectors are viewed as column vectors, and
for a vector $x \in \mathbb{R}^d$, we use $x^{\top}$ to denote
its transpose.
For a stacked vector of $x_i$'s, we often use $(x_1,\ldots,x_m)$ to represent the column vector $[x_1^{\top},\ldots,x_m^{\top}]^{\top}$.
We denote by $\0b$ and $\1b$ the vector of all zeros and ones whose dimensions vary from the context.
The cardinality of a set $S$ is denoted by $|S|$.
We use $I_d$ to denote the identity matrix in $\mathbb{R}^{d\times d}$.
We use $A \otimes B$ for matrices $A \in \mathbb{R}^{n_1\times n_2}$ and $B \in \mathbb{R}^{m_1\times m_2}$ to denote their Kronecker product of size $\mathbb{R}^{n_1m_1\times n_2m_2}$.
For a matrix $A \in \mathbb{R}^{n\times m}$, we use $A_{ij}$ to denote the entry of $i$-th row and $j$-th column.
For any $m \ge 1$, the set of integers $\{1,\ldots,m\}$ is denoted by $[m]$.

\subsection{Problem Formulation}
Consider a multiagent network system whose communication is governed by an undirected graph $\Gc = (\Vc,\Ec)$,
where $\Vc = [m]$ indexes the set of agents,
and $\Ec \subseteq \Vc \times \Vc$ represents the pairs of communicating agents.
If there exists an edge
from agent $i$ to $j$ which we denote by $(i,j)$, agent $i$ may send
its information to agent $j$ and vice versa. Thus, each agent $i \in \Vc$ can directly
receive (resp., send) information only from (resp., to) the agents in its neighborhood
\begin{align}
N_i = \{j \in \Vc \mid (i,j)\in \Ec\} \cup \{i\},
\end{align}
where we assume that there always exists a self-loop $(i,i)$ for
all agents $i \in \Vc$.
Then, the associated Laplacian $L \in \mathbb{R}^{m\times m}$ of $\Gc$ is $L := D-A$
where $D$ is the diagonal degree matrix, 
and $A \in \mathbb{R}^{m\times m}$ is the adjacency matrix with the property that $A_{ij} = 1$
if and only if $(i,j)\in \Ec$ and $i \neq j$, i.e.,
\begin{align}
L_{ij} = \left\{
\begin{array}{ll}
|N_i|-1& \textrm{ if } i = j\\
-1 & \textrm{ if } i\neq j \textrm{ and } (i,j) \in \Ec\\
0 & \textrm{ otherwise.}
\end{array}
\right.
\end{align}

We consider a reformulation of problem \eqref{eqn:orgprob}
which will be used in the development of our decentralized algorithms.
We introduce an individual copy $x_i$ of the decision variable $x$
for each agent $i \in \Vc$
and impose the constraint $x_i = x_j$ for all pairs $(i,j) \in \Ec$.
The transformed problem can be written compactly by using the Laplacian matrix $L$:
\begin{align}\label{eqn:prob}
\min_{\xb} &~F(\xb) := \tsum_{i=1}^m f_i(x_i)\\
\text{s.t. } &~\Lb\xb = \0b, \quad x_i \in  X_i, \text{ for all } i = 1,\ldots,m,\nonumber
\end{align}
where $\xb = (x_1, \ldots, x_m) \in X_1 \times \ldots \times X_m$, $F:X_1 \times \ldots \times X_m \to \mathbb{R}$, and $\Lb = L\otimes I_d \in \mathbb{R}^{md\times md}$.
The constraint $\Lb\xb = \0b$ is a compact way of writing $x_i = x_j$ for all agents $i$ and $j$ which are connected by an edge. By construction, $\Lb$ is symmetric positive semidefinite and its
null space coincides with the ``agreement'' subspace, i.e.,
$\Lb\1b =
\0b$ and $\1b^{\top}\Lb=\0b$.
To ensure each node gets information from every other node, we need the following assumption.
\begin{assumption}\label{assume:G}
The graph $\Gc$ is connected.
\end{assumption}
Under Assumption \ref{assume:G}, problem \eqref{eqn:orgprob} and \eqref{eqn:prob} are equivalent.
We let Assumption \ref{assume:G} be a blanket assumption for the rest of the paper.

We next consider a reformulation of the problem \eqref{eqn:prob} as a saddle point problem.
By the method of Lagrange multipliers,
problem \eqref{eqn:prob} is equivalent to the following saddle point problem:
\begin{align}\label{eqn:saddle}
\min_{\xb \in X^m}   \left[F(\xb) + \max_{\yb\in \mathbb{R}^{md}} \la \Lb\xb, \yb\ra\right],
\end{align}
where $X^m := X_1 \times \ldots \times X_m$ and $\yb = (y_1, \ldots, y_m) \in \mathbb{R}^{md}$ are the Lagrange multipliers
associated with the constraints $\Lb\xb = \0b$.
We assume that there exists an optimal solution $\xb^* \in X^m$
of \eqref{eqn:prob}
and that there exists $\yb^* \in \mathbb{R}^{md}$ such that $(\xb^*,\yb^*)$
is a saddle point of \eqref{eqn:saddle}.

\subsection{Literature review}
Decentralized optimization has been extensively studied in recent years
due to the emergence of large-scale networks.
The seminal work on distributed optimization \cite{Tsiphd,Tsi1986}
has been followed by distributed incremental (sub)gradient methods and proximal methods
\cite{AN2001,Ram:2009,Bertsekas-IP,Wang-Bertsekas},
and more recently the incremental aggregated gradient methods and its proximal variants \cite{Parrilo-aggregated,Bertsekas-aggregated,GLYZ}.
All of these incremental methods are not fully decentralized in a sense that
they require a special star network topology in which the existence of a central authority
is necessary for operation.

To consider a more general network topology, a decentralized
subgradient algorithm was first proposed in \cite{Nedic2009},
and further studied in many other literature
(see e.g. \cite{Duchi-DDA,Martinez-PD,Nedic11,ANAO,tsianos2012consensus}).
These algorithms are intuitive and simple but very slow
due to the fact that they need to use diminishing stepsize rules.
All of these methods require ${\cal O}(1/\epsilon^2)$ inter-node communications and intra-node gradient computations
in order to obtain an $\epsilon$-optimal solution. First-order algorithms by Shi et. al.~\cite{WYin-Extra,WYin-PGExtra}
use constant stepsize rules with backtracking
and require ${\cal O}(1/\epsilon)$ communications 
when the objective function in \eqref{eqn:orgprob} is a relatively simple convex function,
but require both smoothness and strong convexity in order to achieve a linear convergence rate.
Recently, it has been shown in \cite{NaLi-Harness,Wilbur-TV}
that the linear rate of convergence can be obtained for minimizing ``unconstrained'' smooth and strongly convex problems.
These methods do not apply to general nonsmooth and stochastic optimization
problems to be studied in this work.

Another well-known type of decentralized algorithm is based on dual methods
including the distributed dual decomposition \cite{dual-decompos}
and decentralized alternating direction method of multipliers (ADMM) \cite{Wilbur-ADMM,ADMM-linear,Wei-admm}.
The decentralized ADMM \cite{Wilbur-ADMM,ADMM-linear} has been shown to require ${\cal O}(\log 1/\epsilon)$  communications in order to obtain an $\epsilon$-optimal solution under the no constraint, strong convexity and smoothness assumptions while \cite{Wei-admm} has been shown to require ${\cal O}(1/\epsilon)$ communications for relatively simple convex functions $f_i$ (see also \cite{HeJudNem15-1} for the application of mirror-prox method for solving these problems).
These dual-based methods have been further studied via proximal-gradient \cite{InexactConsensus,Chang-Stochastic}.
However, the local Lagrangian minimization problem associated with each agent
cannot be solved efficiently in many cases, especially when the problem is constrained.
Second-order approximation methods \cite{Ribeiro-DQM,Ribeiro-second} have been studied in order to handle this issue,
but due to the nature of these methods differentiability of the objective function is necessary in this case.

There exist some distributed methods that just assume smoothness on the objective functions,
but actually require more communication rounds than gradient computations.
For example, the distributed Nesterov's accelerated gradient method \cite{Jakovetic-Fast} employs multi-consensus
in the inner-loop. Although their method requires ${\cal O}(1/\sqrt{\epsilon})$ intra-node gradient computations,  inter-node communications must increase at a rate of ${\cal O}(\log(k))$ as the iteration $k$ increases.
Similarly, the proximal gradient method with adapt-then-combine (ATC) multi-consensus
strategy and Nesterov's acceleration under the assumption of bounded and Lipschitz gradients \cite{Chen-fastProx} is shown to have ${\cal O}(1/\sqrt{\epsilon})$ intra-node gradient computations, but inter-node communications must increase at a rate of ${\cal O}(k)$.
Due to the nature of decentralized networked systems,
the time required for inter-node communications is higher by a few orders of magnitude than that for intra-node computations.
Multi-consensus schemes in nested loop algorithms
do not account for this feature of networked systems and
hence are less desirable.

Decentralized stochastic optimization methods can be useful when the noisy gradient information
of the function $f_i$, $i = 1,\ldots,m$, in \eqref{eqn:orgprob} is only available or easier to compute.
Stochastic first-order methods for problem \eqref{eqn:orgprob} are studied in
\cite{Duchi-DDA,Ram2010,Nedic11}, all of which require ${\cal O}(1/\epsilon^2)$ inter-node communications
and intra-node gradient computations to obtain an $\epsilon$-optimal solution.
Multiagent mirror descent method for decentralized stochastic optimization \cite{Rabbat-SMD} showed a $\mathcal{O}(1/\epsilon)$ complexity bound when the objective functions are strongly convex.
An alternative form of mirror descent in the multiagent setting was proposed by \cite{Khan-MD} with an asymptotic convergence result.
On a broader scale, decentralized stochastic optimization was also considered in the case of time-varying objective functions in the recent work \cite{DistStoch,Rabbat-online}.
All these previous works in decentralized stochastic optimization suffered from
high communication costs due to the coupled scheme for stochastic subgradient evaluation and communication,
i.e., each evaluation of stochastic subgradient will incur one round of communication.

\subsection{Contribution of the paper}

The main interest of this paper is to develop communication efficient decentralized algorithms
for solving problem \eqref{eqn:prob}
in which $f_i$'s are convex or strongly convex, but not necessarily smooth,
and the local subproblem associated with each agent is nontrivial to solve.
Our contributions in this paper are listed below.

Firstly, we propose a decentralized primal-dual framework which involves only two inter-node communications per iteration.
The proposed method can find an $\epsilon$-optimal solution
both in terms of the primal optimality gap and feasibility residual in $\mathcal{O}(1/\epsilon)$ communication rounds when the objective functions are convex, and the local proximal projection subproblems can be solved exactly.
This algorithm serves as a benchmark in terms of the communication cost for our subsequent development.

Secondly, we  introduce a new decentralized primal-dual type method, called decentralized communication sliding (DCS), where the agents can skip communications while solving their local subproblems iteratively through successive linearizations of their local objective functions.
We show that agents can still find an $\epsilon$-optimal solution in $\mathcal{O}(1/\epsilon)$ (resp., $\mathcal{O}(1/\sqrt{\epsilon})$)
communication rounds while
maintaining the $\mathcal{O}(1/\epsilon^2)$ (resp., $\mathcal{O}(1/\epsilon)$)
bound on the total number of intra-node subgradient evaluations when the objective functions are general convex (resp., strongly convex).
The bounds on the subgradient evaluations are actually comparable to those
optimal complexity bounds required for centralized nonsmooth optimization
under certain conditions on the target accuracy, and hence are not improvable in general.

Thirdly, we present a stochastic decentralized communication sliding method, denoted by SDCS,
for solving stochastic optimization problems and
show complexity bounds similar to those of DCS on the total number of required communication
rounds and stochastic subgradient evaluations. In particular,
only $\mathcal{O}(1/\epsilon)$ (resp., $\mathcal{O}(1/\sqrt{\epsilon})$)
communication rounds are required while agents perform up to
$\mathcal{O}(1/\epsilon^2)$ (resp., $\mathcal{O}(1/\epsilon)$) stochastic
subgradient evaluations for general convex (resp., strongly convex) functions.
Only requiring the access to stochastic
subgradient at each iteration, SDCS is particularly efficient for solving problems with $f_i$ given in the form of \eqnok{exp_prob} and \eqnok{emp_prob}.
In the former case,  SDCS requires only one realization of the random variable at each iteration and provides a
communication-efficient way to deal with streaming data and decentralized machine learning.
In the latter case, each iteration of SDCS requires only one randomly selected component, leading up to a factor
of ${\cal O}(l)$ savings on the total number of subgradient computations over DCS.

To the best of our knowledge, this is the first time that these communication sliding algorithms, and the aforementioned separate complexity bounds
 on communication rounds and (stochastic) subgradient evaluations
are presented in the literature.

\subsection{Organization of the paper}
This paper is organized as follows.
In Section \ref{sec:prelim}, we provide some preliminaries on distance generating functions and prox-functions,
as well as the definition of gap functions, which will be used as termination criteria of our primal-dual methods.
In Section \ref{sec:exact}, we present a new decentralized primal-dual method for solving problem \eqref{eqn:saddle}.
In Section \ref{sec:deterministic}, we present the communication sliding algorithms when the exact subgradients of $f_i$'s
are available and establish their convergence properties for the general and strongly convex case.
In Section \ref{sec:stochastic}, we generalize the algorithms in Section \ref{sec:deterministic} for stochastic problems.
The proofs of the lemmas in Section \ref{sec:exact}-\ref{sec:stochastic} are provided in Section \ref{sec:conv}.
Finally, we provide some concluding remarks in Section \ref{sec:con}.

\setcounter{equation}{0}
\section{Preliminaries}
\label{sec:prelim}
In this section, we provide a brief review on the prox-function,
and define appropriate gap functions which will be used for
the convergence analysis and termination criteria of our primal-dual algorithms.

\subsection{Distance Generating Function and Prox-function}
In this subsection, we define the concept of prox-function,
which is also known as proximity control function or Bregman distance function \cite{BREGMAN1967}.
Prox-function has played an important role
in the recent development of first-order methods for convex programming
as a substantial generalization of the Euclidean projection.
Unlike the standard projection operator $\Pi_U[x] := \argmin_{u\in U}\|x-u\|^2$,
which is inevitably tied to the Euclidean geometry,
prox-function can be flexibly tailored to the geometry of a constraint set $U$.

For any convex set $U$ equipped with an arbitrary norm $\|\cdot\|_U$,
we say that a function $\o: U \to \mathbb{R}$ is a \textit{distance generating function} with modulus $\nu>0$
with respect to $\|\cdot\|_U$, if $\o$ is continuously differentiable and strongly convex with
modulus $\nu$ with respect to $\|\cdot\|_U$, i.e.,
\begin{align}
\la \nabla \o(x) - \nabla \o(u), x-u\ra \ge \nu\|x-u\|_U^2, \quad \forall x, u \in U.
\end{align}

The \textit{prox-function}, or \textit{Bregman distance function}, induced by $\o$ is given by
\begin{align}\label{eqn:def_V}
V(x,u) \equiv V_{\o}(x,u) := \o(u) - [\o(x) + \la \nabla \o(x), u-x\ra].
\end{align}
It then follows from the strong convexity of $\o$ that
\[
V(x,u) \ge \tfrac{\nu}{2}\|x-u\|_U^2, \quad \forall x, u \in U.
\]

We now assume that the individual constraint set $X_i$
for each agent in problem~\eqnok{eqn:orgprob} are equipped with norm $\|\cdot\|_{X_i}$, and
their associated prox-functions are given by $V_i(\cdot,\cdot)$. Moreover, we assume that each $V_i(\cdot,\cdot)$ shares the same strongly convex modulus $\nu=1$, i.e.,
\begin{align}\label{eqn:def_V_i_s}
V_i(x_i,u_i) \ge \tfrac{1}{2}\|x_i-u_i\|_{X_i}^2, \quad \forall x_i, u_i \in X_i, \ i=1,\ldots,m.
\end{align}
We define the norm associated with the primal feasible set $X^m=X_1 \times \ldots \times X_m$ of \eqref{eqn:saddle} as follows:\footnote{
We can define the norm associated with $X^m$ in a more general way,
e.g.,
	$
	\|\xb\|^2:=\tsum_{i=1}^m p_i\|x_i\|_{X_i}^2, \ \forall \xb=(x_1,\ldots,x_m) \in X^m,
	$
for some $p_i > 0$, $i = 1, \ldots,m$. Accordingly,
the prox-function $\mathbf{V}(\cdot,\cdot)$ can be defined  as
	$
	\mathbf{V}(\xb,\ub):=\tsum_{i=1}^m p_iV_i(x_i,u_i), \ \forall \xb,\ub\in X^m.
	$
	This setting gives us flexibility to choose $p_i$'s based on the information of individual $X_i$'s,
	and the possibility to further refine the convergence results. 
}
\beq\label{def_Xm_norm}
\|\xb\|^2\equiv \|\xb\|_{X^m}^2:=\tsum_{i=1}^m \|x_i\|_{X_i}^2,
\eeq
where $\xb=(x_1,\ldots,x_m)\in X^m$ for any $x_i\in X_i$. Therefore, the corresponding prox-function $\mathbf{V}(\cdot,\cdot)$ can be defined as
\beq\label{eqn:def_Vb}
\mathbf{V}(\xb,\ub):=\tsum_{i=1}^m V_i(x_i,u_i), \ \forall \xb,\ub\in X^m.
\eeq
Note that by \eqref{eqn:def_V_i_s} and \eqref{def_Xm_norm}, it can be easily seen that
\beq\label{eqn:def_V_s}
\mathbf{V}(\xb,\ub)\ge \tfrac{1}{2}\|\xb-\ub\|^2, \ \forall \xb, \ub \in X^m.
\eeq

Throughout the paper, we endow the dual space where the multipliers $\yb$ of \eqnok{eqn:saddle} reside
with the standard Euclidean norm $\|\cdot\|_2$, since the feasible region of $\yb$ is unbounded.
For simplicity, we often write $\|\yb\|$ instead of $\|\yb\|_2$ for a dual multiplier $\yb \in  \mathbb{R}^{md}$.

\subsection{Gap Functions: Termination Criteria}
Given a pair of feasible solutions $\zb = (\xb,\yb)$  and $\bar{\zb} = (\bar{\xb},\bar{\yb})$ of \eqref{eqn:saddle},
we define the \textit{primal-dual gap function} $Q(\zb;\bar{\zb})$ by
\begin{align}\label{eqn:gap}
Q(\zb;\bar{\zb}) := &~F(\xb) + \la \Lb\xb, \bar{\yb} \ra - [F(\bar{\xb}) + \la \Lb\bar{\xb}, \yb \ra].
\end{align}
Sometimes we also use the notations $Q(\zb;\bar{\zb}) := Q(\xb, \yb;\bar{\xb},\bar{\yb})$
or $Q(\zb;\bar{\zb}) := Q(\xb,\yb;\bar{\zb}) = Q(\zb;\bar{\xb},\bar{\yb})$.
One can easily see that $Q(\zb^*;\zb) \le 0$ and $Q(\zb;\zb^*) \ge 0$ for all $\zb \in X^m \times \mathbb{R}^{md}$, where $\zb^* = (\xb^*,\yb^*)$ is a saddle point of \eqref{eqn:saddle}.
For compact sets $X^m \subset \mathbb{R}^{md}$, $Y \subset \mathbb{R}^{md}$, the gap function
\begin{align} \label{eqn:gapQ0}
\sup_{\bar{\zb}\in X^m \times Y} Q(\zb;\bar{\zb})
\end{align}
measures the accuracy of the approximate solution $\zb$ to the saddle point problem \eqref{eqn:saddle}.

However, the saddle point formulation \eqref{eqn:saddle} of our problem of interest \eqref{eqn:orgprob}
may have an unbounded feasible set.
We adopt the perturbation-based termination criterion by Monteiro and Svaiter \cite{Monteiro01,Monteiro02,Monteiro03}
and propose a modified version of the gap function in \eqref{eqn:gapQ0}.
More specifically, we define
\begin{align}\label{eqn:gapp}
g_Y(\vb,\zb) := \sup_{\bar{\yb}\in Y} Q(\zb;\xb^*,\bar{\yb}) - \la \vb, \bar{\yb}\ra,
\end{align}
for any closed set $Y \subseteq \mathbb{R}^{md}$, $\zb \in X^m\times \mathbb{R}^{md}$ and $\vb \in \bbr^{md}$.
If $Y = \mathbb{R}^{md}$, we omit the subscript $Y$ and simply use the notation $g(\vb,\zb)$.

This perturbed gap function allows us to bound the objective function value and the feasibility separately.
We first define the following terminology.
\begin{definition}
A point $\xb \in X^m$ is called an $(\epsilon,\delta)$-solution of \eqref{eqn:prob} if
\begin{align}
F(\xb) - F(\xb^*) \le \epsilon \text{ and } \|\Lb\xb\|\le \delta.
\end{align}
We say that $\xb$ has primal residual $\epsilon$ and feasibility residual $\delta$.
\end{definition}
Similarly, a stochastic $(\epsilon,\delta)$-solution of \eqref{eqn:prob} can be defined as a point $\hat \xb \in X^m$ s.t. $\bbe[F(\hat \xb)-F(\xb^*)]\le \epsilon$ and $\bbe[\|\Lb\hat \xb\|]\le \delta$ for some $\epsilon,\delta>0$.
Note that for problem \eqref{eqn:prob}, the feasibility residual measures the disagreement among the local copies $x_i$, for $i \in \Vc$.

In the following proposition, we adopt a result from \cite[Proposition 2.1]{GL-AADMM} to describe the relationship between the perturbed gap function \eqref{eqn:gapp} and the approximate solutions to problem \eqref{eqn:prob}.
Although the proposition was originally developed for deterministic cases,
the extension of this to stochastic cases is straightforward.
\begin{proposition}\label{prop:approx}
For any $Y \subset \mathbb{R}^{md}$ such that $\0b \in Y$,
if $g_Y(\Lb\xb,\zb) \le \epsilon < \infty$ and $\|\Lb\xb\| \le \delta$, where $\zb = (\xb,\yb) \in X^m \times \mathbb{R}^{md}$, then $\xb$ is an $(\epsilon,\delta)$-solution of \eqref{eqn:prob}.
In particular, when $Y = \mathbb{R}^{md}$, for any $\vb$ such that $g(\vb,\zb) \le \epsilon < \infty$
and $\|\vb\|\le \delta$, we always have $\vb = \Lb\xb$.
\end{proposition}

\setcounter{equation}{0}
\section{Decentralized Primal-Dual}
\label{sec:exact}

In this section, we describe an algorithmic framework
for solving the saddle point problem \eqref{eqn:saddle} in a decentralized fashion.
The basic scheme of the decentralized primal-dual method in Algorithm~\ref{alg:DPD}
is similar to Chambolle and Pork's primal-dual method in \cite{Chambolle-PD}.
The primal-dual method in \cite{Chambolle-PD} is
an efficient and simple method for solving saddle point problems, which
 can be viewed as a refined version of the primal-dual hybrid gradient method
by Arrow et al.~\cite{arrow1958studies}. However, its design and analysis is more closely related to a few recent important works
which established the ${\cal O}(1/k)$ rate of convergence for solving bilinear saddle point problems (e.g.,~\cite{Nest05-1,Nem05-1,MonSva10-1,he2012on}).
Recently, Chen, Lan and Ouyang~\cite{CheLanOu13-1} incorporated
Bregman distance into the primal-dual method together with an acceleration step.
Dang and Lan~\cite{Dang-Lan}, and Chambolle and Pork~\cite{ChamPoc14-1} discussed improved algorithms
for problems with strongly convex primal or dual functions. Randomized
versions of the primal-dual method have been discussed by
Zhang and Xiao~\cite{Yuchen14}, and Dang and Lan~\cite{Dang-Lan}.
Lan and Zhou~\cite{GLYZ} revealed some inherent relationship between
Nesterov's accelerated gradient method and the primal-dual method, and
presented an optimal randomized incremental gradient method.

Our main goals here in this section are to: 1) adapt the primal-dual framework for a decentralized setting;
and 2) provide complexity results (number of communication
rounds and subgradient computations) separately in terms of primal functional optimality gap and constraint (or consistency) violation.
It should be stressed that the main contributions of this paper exist in
the development of decentralized communication sliding algorithms (see Section \ref{sec:deterministic} and \ref{sec:stochastic}).
However, introducing the basic decentralized primal-dual method here will help us better explain these methods
and provide us with a certain benchmark in terms of the communication cost.

\begin{algorithm}
\caption{Decentralized primal-dual}
\label{alg:DPD}
\begin{algorithmic}
\STATE{Let $\xb^0 = \xb^{-1} \in X^m$ and $\yb^0 \in \mathbb{R}^{md}$, the nonnegative parameters $\{\a_k\}$, $\{\tau_k\}$ and $\{\eta_k\}$, and the weights $\{\t_k\}$ be given.}
\FOR{$k = 1, \ldots, N$}
\STATE{Update $\zb^k = (\xb^k,\yb^k)$ according to}
\begin{align}
\tilde{\xb}^k =&~ \a_k (\xb^{k-1}-\xb^{k-2}) + \xb^{k-1}\label{eqn:algo1}\\
\yb^k =&~ \argmin_{\yb \in \mathbb{R}^{md}}~  \la -\Lb\tilde{\xb}^k, \yb\ra + \tfrac{\tau_k}{2}\|\yb-\yb^{k-1}\|^2\label{eqn:algo2}\\
\xb^k = &~ \argmin_{\xb \in X^m}~  \la \Lb\yb^k, \xb\ra + F(\xb) + \eta_k\mathbf{V}(\xb^{k-1},\xb)\label{eqn:algo3}
\end{align}
\ENDFOR
\RETURN $\bar{\zb}^N = (\tsum_{k=1}^N \t_k)^{-1} \tsum_{k=1}^N \t_k \zb^{k}$.
\end{algorithmic}
\end{algorithm}

\begin{algorithm}
\caption{Decentralized primal-dual update for each agent $i$}
\label{alg:DPD-i}
\begin{algorithmic}
\STATE{Let $x_i^0 = x_i^{-1}\in X_i$ and $y_i^0 \in \mathbb{R}^d$ for $i \in [m]$, the nonnegative parameters $\{\a_k\}$, $\{\tau_k\}$ and $\{\eta_k\}$, and the weights $\{\t_k\}$ be given.}
\FOR{$k = 1, \ldots, N$}
\STATE{Update $z_i^k = (x_i^k,y_i^k)$ according to}
\begin{align}
\tilde{x}_i^k =&~ \a_k (x_i^{k-1}-x_i^{k-2}) + x_i^{k-1}\label{eqn:algo1-i}\\
v_i^k = &~\tsum_{j\in N_i}L_{ij}\tilde{x}_j^k\label{eqn:algo2-i}\\
y_i^k =&~ y_i^{k-1} +\tfrac{1}{\tau_k}v_i^k\label{eqn:algo3-i}\\
w_i^k = &~\tsum_{j\in N_i}L_{ij}y_j^k\label{eqn:algo4-i}\\
x_i^k = &~ \argmin_{x_i\in X_i}~  \la w_i^k, x_i\ra + f_i(x_i) + \eta_kV_i(x_i^{k-1},x_i)\label{eqn:algo5-i}
\end{align}
\ENDFOR
\RETURN $\bar{\zb}^N = (\tsum_{k=1}^N \t_k)^{-1} \tsum_{k=1}^N \t_k \zb^{k}$
\end{algorithmic}
\end{algorithm}

\subsection{The Algorithm}
The primal-dual algorithm in Algorithm~\ref{alg:DPD} can be decentralized due to the structure of the Laplacian $\Lb$.
Recalling that $\xb = (x_1,\ldots,x_m)$ and $\yb = (y_1,\ldots,y_m)$,
each agent $i$'s local update rule can be separately written as in Algorithm~\ref{alg:DPD-i}.
Each agent $i$ maintains two local sequences, namely, the primal estimates $\{x_i^k\}$
and the dual variables $\{y_i^k\}$.
The element $x_i^k$ can be seen as agent $i$'s estimate of the decision variable $x$ at time $k$,
while $y_i^k$ is a subvector of all dual variables $\yb^k$
associated with the agent $i$'s consistency constraints with its neighbors.

More specifically, each primal estimate $x_i^0$ is locally initialized
from some arbitrary point in $X_i$, and $x_i^{-1}$ is also set to be the same value.
At each time step $k \ge 1$, each agent $i \in \Vc$
computes a local
prediction $\tilde{x}_i^k$ using the two previous primal estimates (ref. \eqref{eqn:algo1-i}),
and broadcasts this to all of the nodes in its neighborhood, i.e., to all agents $j \in N_i$.
In \eqref{eqn:algo2-i}-\eqref{eqn:algo3-i},
each agent $i$ calculates the neighborhood disagreement $v_i^k$ using the messages received from agents in $N_i$,
and updates the dual subvector $y_i^k$.
Then, another round of communication occurs in \eqref{eqn:algo4-i} to broadcast this updated dual variables
and calculate $w_i^k$. Therefore, each iteration $k$ involves two communication rounds,
one for the primal estimates and the other for the dual variables.
Lastly, each agent $i$ solves the proximal projection subproblem \eqref{eqn:algo5-i}.
Note that the description of the algorithm is only conceptual at this moment
since we have not specified the parameters $\{\a_k\}$, $\{\tau_k\}$, $\{\eta_k\}$ and $\{\t_k\}$ yet.
We will later instantiate this generic algorithm when we state its convergence properties.

\subsection{Convergence of the Decentralized Primal-dual Method}
For the sake of simplicity, we focus only on the case when $f_i$'s are general convex functions in this section.
We leave the discussion about the case when $f_i$'s are strongly convex later
in Sections~\ref{sec:deterministic} and \ref{sec:stochastic} for decentralized
communication sliding algorithms.

In the following lemma, we present estimates on the gap function defined in \eqref{eqn:gap}
together with conditions on the parameters $\{\a_k\}$, $\{\tau_k\}$, $\{\eta_k\}$ and $\{\t_k\}$,
which will be used to provide the rate of convergence for the decentralized primal-dual method.
The proof of this lemma can be found in Section \ref{sec:conv}.

\begin{lemma}\label{lem:main_exact}
Let the iterates $\zb^k = (\xb^k,\yb^k)$, $k = 1, \ldots, N$ be generated by Algorithm~\ref{alg:DPD}
and $\bar{\zb}^N$ be defined as $\bar{\zb}^N := \left(\tsum_{k=1}^N\theta_k\right)^{-1}\tsum_{k=1}^N\theta_k\zb^k$.
	Assume that the parameters $\{\a_k\}$,  $\{\tau_k\}$, $\{\eta_k\}$ and $\{\t_k\}$ in Algorithm~\ref{alg:DPD} satisfy
	\begin{align}
	\t_{k}\eta_{k} \le &~\t_{k-1}\eta_{k-1}, \quad k = 2, \ldots, N,\label{theta_eta}\\
	\a_{k}\t_{k}=&~\t_{k-1}, \quad k = 2, \ldots, N,\label{alpha_theta}\\
	\t_{k}\tau_{k}\le &~\t_{k-1}\tau_{k-1}, \quad k = 2, \ldots, N,\label{theta_tau}\\
	\a_k \|\Lb\|^2 \le &~\eta_{k-1}\tau_k, \quad k = 2, \ldots, N,\label{eta_tau_L_k}\\
	\theta_1\tau_1 = &~\theta_N\tau_N,\label{eta_tau}\\
	\theta_N\|\Lb\|^2 \le &~ \theta_1\tau_1\eta_N \label{eta_tau_theta}.
	\end{align}
	Then, for any $\zb:=(\xb,\yb)\in X^m\times \bbr^{md}$, we have
	\begin{align}\label{outer_recursion}
	Q(\bar \zb^N;\zb)
    \le \left(\tsum_{k=1}^N \theta_k\right)^{-1}\left[\theta_1\eta_1\mathbf{V}(\xb^0,\xb)+\tfrac{\theta_1\tau_1}{2}\|\yb^0\|^2 + \la \vb,\yb\ra\right],
	\end{align}
	where $Q$ is defined in \eqref{eqn:gap} and $\vb$ is defined as
    \begin{align}\label{eqn:vb}
    \vb := \theta_N\Lb(\xb^N-\xb^{N-1}) + \theta_1\tau_1(\yb^N-\yb^0).
    \end{align}
    Furthermore, for any saddle point $(\xb^*,\yb^*)$ of \eqref{eqn:saddle}, we have
    \begin{align}\label{xNs_bnd_pd}
    \tfrac{\theta_N}{2}\left(1-\tfrac{\|\Lb\|^2}{\eta_N\tau_N}\right)\max\{\eta_N\|\xb^{N-1}-\xb^N\|^2,\tau_N\|\yb^*-\yb^N\|^2\}&\le \theta_1\eta_1\mathbf{V}(\xb^0,\xb^*)+\tfrac{\theta_1\tau_1}{2}\|\yb^*-\yb^0\|^2.
    \end{align}
\end{lemma}

In the following theorem, we provide a specific selection of $\{\alpha_k\}$, $\{\tau_k\}$, $\{\eta_k\}$ and $\{\theta_k\}$ satisfying \eqref{theta_eta}-\eqref{eta_tau_theta}. Using Lemma \ref{lem:main_exact} and Proposition \ref{prop:approx}, we also establish the complexity of the
decentralized primal-dual method for computing an $(\epsilon,\delta)$-solution of problem \eqref{eqn:prob}.
\begin{theorem}\label{exact_thm}
	Let $\xb^*$ be a saddle point of \eqref{eqn:prob}, and suppose that $\{\alpha_k\}$, $\{\tau_k\}$, $\{\eta_k\}$ and $\{\theta_k\}$ are set to
	\beq\label{exact_para}
	\alpha_k=\theta_k = 1,~ \eta_k=2\|\Lb\|,  \mbox{ and } \tau_k=\|\Lb\|,\quad \forall k=1, \dots, N.
	\eeq
	Then, for any $N \ge 1$, we have
\begin{align}\label{eqn:exactobj}
	F(\bar\xb^N)- F(\xb^*)\le \tfrac{\|\Lb\|}{N}\left[2\mathbf{V}(\xb^0,\xb^*)+\tfrac{1}{2}\|\yb^0\|^2\right]
\end{align}
and
\begin{align}\label{eqn:exactfeas}
\|\Lb\bar\xb^N\| \le \tfrac{2\|\Lb\|}{N}\left[ 3\sqrt{\mathbf{V}(\xb^0,\xb^*)} + 2\|\yb^*-\yb^0\|\right],
\end{align}
where $\bar\xb^N = \tfrac{1}{N}\tsum_{k=1}^N\xb^k$.
\end{theorem}
\begin{proof}
	It is easy to check that \eqnok{exact_para} satisfies conditions \eqnok{theta_eta}-\eqnok{eta_tau_theta}. Therefore,
by plugging these values in \eqnok{outer_recursion}, we have
	\beq\label{eqn:Qpb}
	Q(\bar \zb^N;\xb^*,\yb)\le \tfrac{1}{N}\left[2\|\Lb\|\mathbf{V}(\xb^0,\xb^*)+\tfrac{\|\Lb\|}{2}\|\yb^0\|^2\right]+\tfrac{1}{N}\la \vb,\yb \ra.
	\eeq
Letting $\vb^N := \tfrac{1}{N}\vb$, then from \eqref{eqn:vb} and \eqref{xNs_bnd_pd} we have
\begin{align*}
\|\vb^N\| &\le  \tfrac{\|\Lb\|}{N}\left[\|\xb^N-\xb^{N-1}\| + \|\yb^N-\yb^*\|+\|\yb^*-\yb^0\|\right]\\
&\le\tfrac{\|\Lb\|}{N}\left[3\sqrt{4\mathbf{V}(\xb^0,\xb^*)+\|\yb^*-\yb^0\|^2}+\|\yb^*-\yb^0\|\right].
\end{align*}
Furthermore, by \eqref{eqn:Qpb} we have
\begin{align*}
	g(\vb^N,\bar\zb^N)\le \tfrac{\|\Lb\|}{N}\left[2\mathbf{V}(\xb^0,\xb^*)+\tfrac{1}{2}\|\yb^0\|^2\right].
\end{align*}
The results in \eqref{eqn:exactobj} and \eqref{eqn:exactfeas} then immediately follow from Proposition \ref{prop:approx} and the
above two inequalities.
\end{proof}

From \eqref{eqn:exactobj}-\eqref{eqn:exactfeas}, we can see that the complexity of decentralized primal-dual method for computing
an $(\epsilon,\delta)$-solution is ${\cal O}(1/\epsilon)$ for the primal functional optimality and ${\cal O}(1/\delta)$ for the constraint violation. Since each iteration involves a constant number of communication rounds,
the number of inter-node communications required is also in the same order.

\setcounter{equation}{0}
\section{Decentralized Communication Sliding}
\label{sec:deterministic}

In this section, we present a new decentralized primal-dual type method, namely,
the decentralized communication sliding (DCS) method
for the case when the primal subproblem \eqref{eqn:algo5-i} is not easy to solve.
We show that one can still maintain the same number of inter-node communications
even when the subproblem \eqref{eqn:algo5-i} is approximately solved through an iterative
subgradient descent procedure, and that
the total number of required subgradient evaluations is comparable to
centralized mirror descent methods. Throughout this section, we consider the deterministic case where exact subgradients of $f_i$'s are available.

\begin{algorithm}
\caption{Decentralized Communication Sliding (DCS)}
\label{alg:DCS}
\begin{algorithmic}
\STATE{Let $x_i^0 = x_i^{-1} = \hat{x}_i^0 \in X_i$, $y_i^0 \in \mathbb{R}^d$ for $i \in [m]$ and the nonnegative parameters $\{\a_k\}$, $\{\tau_k\}$, $\{\eta_k\}$ and $\{T_k\}$ be given.}
\FOR{$k = 1, \ldots, N$}
\STATE{Update $\zb^k = (\hat{\xb}^k,\yb^k)$ according to}
\begin{align}
\tilde{x}_i^k =&~ \a_k (\hat{x}_i^{k-1}-x_i^{k-2}) + x_i^{k-1}\label{eqn:algo1-cs}\\
v_i^k = &~\tsum_{j\in N_i}L_{ij}\tilde{x}_j^k\label{eqn:algo2-cs}\\
y_i^k =&~ \argmin_{y_i \in \mathbb{R}^d}~  \la -v_i^k, y_i\ra + \tfrac{\tau_k}{2}\|y_i-y_i^{k-1}\|^2= y_i^{k-1} +\tfrac{1}{\tau_k}v_i^k\label{eqn:algo3-cs}\\
w_i^k = &~\tsum_{j\in N_i}L_{ij}y_j^k\label{eqn:algo4-cs}\\
(x_i^k,\hat{x}_i^k) =&~ \text{CS}(f_i,X_i,V_i,T_k,\eta_k,w_i^k,x_i^{k-1})\label{eqn:algo5-cs}
\end{align}
\ENDFOR
\RETURN $z_i^N = (\hat{x_i}^N,y_i^N)$

\STATE
\STATE The CS (Communication-Sliding) procedure called at \eqref{eqn:algo5-cs} is stated as follows.\\
\textbf{procedure:} $(x,\hat{x}) =\text{CS}(\phi,U,V,T,\eta,w,x)$
\STATE Let $u^0 = \hat{u}^0 = x$ and the parameters $\{\b_t\}$ and $\{\lambda_t\}$ be given.
\FOR{$t = 1,\ldots,T$}
\STATE
\begin{align}
h^{t-1} = &~  \phi'(u^{t-1}) \in \partial \phi(u^{t-1}) \label{eqn:inner0}\\
u^t = &~ \argmin_{u \in U}\left[\la w + h^{t-1}, u \ra +\eta V(x,u) + \eta\b_tV(u^{t-1},u)\right]\label{eqn:inner1}
\end{align}
\ENDFOR
\STATE Set
\begin{align}
\hat{u}^T := &~ \left(\tsum_{t = 1}^T\lambda_t\right)^{-1}\tsum_{t=1}^T \lambda_tu^t. \label{eqn:inner2}
\end{align}
\STATE Set $x = u^T$ and $\hat{x} = \hat{u}^T$.\\
\textbf{end procedure}
\end{algorithmic}
\end{algorithm}

\subsection{The DCS Algorithm}
We formally describe our DCS algorithm in Algorithm~\ref{alg:DCS}.
We say that an outer iteration of the DCS algorithm, which we call the outer-loop, occurs whenever the index $k$ in
Algorithm~\ref{alg:DCS} is incremented by 1.
Since the subproblems are solved inexactly, the outer-loop of the primal-dual algorithm also needs to be modified in order to attain the best possible rate of convergence. 
In particular, in addition to the primal estimate $\{x_i^k\}$, we let each agent $i$ maintain another
primal sequence $\{\hat{x}_i^k\}$ (cf. the definition of $\tilde x_i^k$ in \eqref{eqn:algo1-cs}), which will
later play a crucial role in the development and convergence proof of the algorithm.
Observe that the DCS method, in spirit, has been inspired by some of our recent work on gradient sliding \cite{Lan-Sliding}.
However, the gradient sliding method in \cite{Lan-Sliding} focuses on how to save gradient evaluations
for solving certain structured convex optimization problems,
rather than how to save communication rounds for decentralized optimization, and its algorithmic scheme is also quite different
from the DCS method.

The steps \eqref{eqn:algo1-cs}-\eqref{eqn:algo4-cs} are similar to those in Algorithm~\ref{alg:DPD-i}
except that the local prediction $\tilde{x}_i^k$ in \eqref{eqn:algo1-cs}
is computed using the two previous primal estimates $\hat{x}_i^{k-1}$ and $x_i^{k-1}$.
The CS procedure in \eqref{eqn:algo5-cs}, which we call the inner loop, solves the subproblem \eqref{eqn:algo5-i} iteratively for $T_k$ iterations.
Each inner loop iteration consists of the computation of the subgradient $f_i(u^{t-1})$ in \eqref{eqn:inner0}
and the solution of the projection subproblem in \eqref{eqn:inner1},
which is assumed to be relatively easy to solve.
Note that the description of the algorithm is only conceptual at this moment
since we have not specified the parameters $\{\a_k\}$, $\{\eta_k\}$, $\{\tau_k\}$, $\{T_k\}$ $\{\b_t\}$
and $\{\lambda_t\}$ yet. We will later instantiate this generic algorithm when we state its convergence properties.

A few remarks about this algorithm are in order.
Firstly, a critical difference of this routine compared to the exact version (Algorithm~\ref{alg:DPD-i}) is that
one needs to compute a pair of approximate solutions $x_i^k$ and $\hat{x}_i^k$.
While both $x_i^k$ and $\hat{x}_i^k$
can be seen as agent $i$'s estimate of the decision variable $x$ at time $k$,
$x_i^k$ will be used to define the subproblem \eqref{eqn:inner1} for the next call to the CS procedure
and $\hat{x}_i^k$ will be used to produce a weighted sum of all the inner loop iterates.
Secondly, since the same $w_i^k$
has been used throughout the $T_k$ iterations of the CS procedure,
no additional communications of the dual variables are required
when performing the subgradient projection step \eqref{eqn:inner1} for $T_k$ times.
This differs from the accelerated gradient methods in \cite{Chen-fastProx,Jakovetic-Fast}
where the number of inter-node communications at each iteration $k$ increase linearly or sublinearly in the order of $k$.

Note that the results of the CS procedure at iteration $k$ for agents $i \in \Vc$ collectively generate a pair
of approximate solutions $\hat{\xb}^k = (\hat{x}_1^k,\ldots,\hat{x}_m^k)$
and $\xb^k = (x_1^k,\ldots,x_m^k)$ to the proximal projection subproblem \eqref{eqn:algo3}.
For later convenience, we refer to the subproblem at iteration $k$ as $\Phi^k(\xb)$, i.e.,
\begin{align}\label{eqn:defPhi_k}
\argmin_{\xb \in X^m} \left\{\Phi^k(\xb) := \la \Lb\yb^k, \xb\ra + F(\xb) + \eta_k\mathbf{V}(\xb^{k-1},\xb)\right\}.
\end{align}

\subsection{Convergence of DCS on General Convex Functions}
We now establish the main convergence properties of the DCS algorithm.
More specifically, we provide in Lemma \ref{lem:main_deter}
an estimate on the gap function defined in \eqref{eqn:gap}
together with stepsize policies which work for the general nonsmooth convex case with $\mu = 0$ (cf. \eqref{eqn:nonsmooth}).
The proof of this lemma can be found in Section \ref{sec:conv}.

\begin{lemma}\label{lem:main_deter}
Let the iterates $(\hat{\xb}^k,\yb^k)$, $k = 1, \ldots, N$ be generated by Algorithm~\ref{alg:DCS}
and $\hat{\zb}^N$ be defined as
$\hat{\zb}^N := \left(\tsum_{k=1}^N\theta_k\right)^{-1}\tsum_{k=1}^N\theta_k(\hat \xb^k,\yb^k)$.
Assume that the objective $f_i$, $i = 1, \ldots, m$, are general nonsmooth convex functions, i.e., $\mu = 0$ and $M >0$.
Let the parameters $\{\a_k\}$, $\{\t_k\}$, $\{\eta_k\}$, $\{\tau_k\}$ and $\{T_k\}$ in Algorithm~\ref{alg:DCS} satisfy
\eqref{alpha_theta}-\eqref{eta_tau_theta} and
\begin{align}\label{theta_eta_d}
\theta_k\tfrac{(T_k+1)(T_k+2)\eta_k}{T_k(T_k+3)} \le \theta_{k-1}\tfrac{(T_{k-1}+1)(T_{k-1}+2)\eta_{k-1}}{T_{k-1}(T_{k-1}+3)},\quad k = 2,\ldots, N.
\end{align}
Let the parameters $\{\lambda_t\}$ and $\{\beta_t\}$ in the CS procedure of Algorithm~\ref{alg:DCS} be set to
\begin{align}\label{eqn:csparam}
\lambda_t = t+1, \quad \beta_t = \tfrac{t}{2}, \quad \forall t \ge1.
\end{align}
Then, we have for all $\zb \in X^m\times \bbr^{md}$,
	\begin{align}\label{eqn:Q0}
	Q(\hat{\zb}^N;\zb)
	&\le \left(\tsum_{k=1}^N \theta_k\right)^{-1}
    \left[\tfrac{(T_1+1)(T_1+2)\theta_1\eta_1}{T_1(T_1+3)}\mathbf{V}(\xb^0,\xb)+\tfrac{\theta_1\tau_1}{2}\|\yb^0\|^2 + \la\hat\vb,\yb\ra
	+ \tsum_{k=1}^N\tfrac{4mM^2\theta_k}{(T_k+3)\eta_k}\right],
	\end{align}
where $\hat\vb := \theta_N \Lb (\hat \xb^N - \xb^{N-1}) + \theta_1\tau_1 (\yb^N - \yb^0)$
and $Q$ is defined in \eqref{eqn:gap}.
	Furthermore, for any saddle point $(\xb^*,\yb^*)$ of \eqref{eqn:saddle}, we have
	\begin{align}\label{xNs_bnd_ns}
	\tfrac{\theta_N}{2}\left(1-\tfrac{\|\Lb\|^2}{\eta_N\tau_N}\right)&\max\{\eta_N\|\hat\xb^N-\xb^{N-1}\|^2,\tau_N\|\yb^*-\yb^N\|^2\}\\\nn
	&~~~\le \tfrac{(T_1+1)(T_1+2)\theta_1\eta_1}{T_1(T_1+3)}\mathbf{V}(\xb^0, \xb^*)
	+\tfrac{\theta_1\tau_1}{2}\|\yb^*-\yb^0\|^2
	+ \tsum_{k=1}^N\tfrac{4mM^2\theta_k}{\eta_k(T_k+3)}.
	\end{align}
\end{lemma}

In the following theorem, we provide a specific selection of $\{\alpha_k\}$, $\{\theta_k\}$, $\{\eta_k\}$, $\{\tau_k\}$ and $\{T_k\}$  satisfying \eqref{alpha_theta}-\eqref{eta_tau_theta} and \eqref{theta_eta_d}. Using Lemma \ref{lem:main_deter} and Proposition \ref{prop:approx}, we also establish the complexity of the DCS method for computing an $(\epsilon,\delta)$-solution of problem \eqref{eqn:prob} when the objective functions are general convex.
	\begin{theorem}\label{deter_thm_ns}
		Let $\xb^*$ be an optimal solution of \eqnok{eqn:prob}, the parameters $\{\lambda_t\}$ and $\{\beta_t\}$ in the CS procedure of Algorithm~\ref{alg:DCS} be set to
		\eqref{eqn:csparam}, and suppose that $\{\alpha_k\}$, $\{\theta_k\}$, $\{\eta_k\}$, $\{\tau_k\}$ and $\{T_k\}$ are set to
		\beq\label{deter_ns_para}
		\alpha_k=\theta_k = 1,~\eta_k=2\|\Lb\|,~\tau_k=\|\Lb\|,\mbox{ and }T_k=\left\lceil \tfrac{m M^2N}{\|\Lb\|^2\tilde D}\right\rceil,
		\quad \forall k=1,\ldots, N,
		\eeq
		for some $\tilde D>0$. Then, for any $N \ge 1$, we have
		\begin{align}\label{eqn:deterobj}
		F(\hat\xb^N)- F(\xb^*) \le \tfrac{\|\Lb\|}{N}\left[3\mathbf{V}(\xb^0,\xb^*)+\tfrac{1}{2}\|\yb^0\|^2+2\tilde D\right]
		\end{align}
		and
		\begin{align}\label{eqn:deterfeas}
		\|\Lb\hat\xb^N\| \le  \tfrac{\|\Lb\|}{N}\left[ 3\sqrt{6\mathbf{V}(\xb^0,\xb^*)+4\tilde D} + 4\|\yb^* - \yb^0\|\right],
		\end{align}
		where $\hat\xb^N = \tfrac{1}{N}\tsum_{k=1}^N \hat\xb^k$.
	\end{theorem}
	\begin{proof}
		It is easy to check that \eqnok{deter_ns_para} satisfies conditions \eqnok{alpha_theta}-\eqnok{eta_tau_theta}
		and \eqref{theta_eta_d}. Particularly,
		\[
		\tfrac{(T_1+1)(T_1+2)}{T_1(T_1+3)}=1+\tfrac{2}{T_1^2+3T_1}\le \tfrac{3}{2}.
		\]
		Therefore, by plugging in these values to \eqnok{eqn:Q0}, we have
		\begin{align}\label{eqn:Qpbcs}
		Q(\hat \zb^N;\xb^*,\yb)\le \tfrac{\|\Lb\|}{N}\left[3\mathbf{V}(\xb^0,\xb^*)+\tfrac{1}{2}\|\yb^0\|^2+2\tilde D\right]+\tfrac{1}{N}\la \hat\vb,\yb\ra.
		\end{align}
		Letting $\hat\vb^N = \tfrac{1}{N}\hat\vb$, then from \eqref{xNs_bnd_ns}, we have
		\begin{align*}
		\|\hat\vb^N\| &\le  \tfrac{\|\Lb\|}{N}\left[ \|\hat \xb^N - \xb^{N-1}\| + \|\yb^N - \yb^*\|+\|\yb^*-\yb^0\|\right]\\
		&\le \tfrac{\|\Lb\|}{N}\left[3\sqrt{6\mathbf{V}(\xb^0,\xb^*)+\|\yb^*-\yb^0\|^2+4\tilde D} + \|\yb^* - \yb^0\|\right].
		\end{align*}
		Furthermore, by \eqref{eqn:Qpbcs}, we have
		\begin{align*}
		g(\hat\vb^N,\hat \zb^N) \le \tfrac{\|\Lb\|}{N}\left[3\mathbf{V}(\xb^0,\xb^*)+\tfrac{1}{2}\|\yb^0\|^2+2\tilde D\right].
		\end{align*}
		Applying Proposition \ref{prop:approx} to the above two inequalities,
		the results in \eqref{eqn:deterobj} and \eqref{eqn:deterfeas} follow immediately.
	\end{proof}

We now make some remarks about the results obtained in Theorem~\ref{deter_thm_ns}.
Firstly, even though one can choose any $\tilde D > 0$ (e.g., $\tilde D = 1$) in \eqnok{deter_ns_para},
the best selection of $\tilde D$ would be $\mathbf{V}(\xb^0,\xb^*)$ so that the first and third terms in \eqnok{eqn:Qpbcs}
are about the same order. In practice, if there exists an estimate ${\cal D}_{X^m}>0$ s.t.
		\beq\label{V_bndness}
		\mathbf{V}(\xb_1,\xb_2)\le {\cal D}_{X^m}^2, \, \forall \xb_1,\xb_2\in X^m,
		\eeq
then we can set $\tilde D={\cal D}_{X^m}^2$.

Secondly, the complexity of the DCS method directly follows from \eqnok{eqn:deterobj} and \eqnok{eqn:deterfeas}.
For simplicity, let us assume that $X$ is bounded,  $\tilde D={\cal D}_{X^m}^2$
and $\yb^0=\0b$. We can see that the total number of inter-node communication rounds and intra-node subgradient evaluations required by each
agent for finding an $(\epsilon, \delta)$-solution of \eqref{eqn:prob}
can be bounded by
		\beq\label{complexity_ns}
        {\cal O}\left\{\|\Lb\|\max\left(\tfrac{{\cal D}_{X^m}^2}{\epsilon},\tfrac{{\cal D}_{X^m}+\|\yb^*\|}{\delta}\right)\right\}\ \
        \mbox{and} \ \
        {\cal O}\left\{mM^2\max\left(\tfrac{{\cal D}_{X^m}^2}{\epsilon^2},\tfrac{{\cal D}_{X^m}^2+\|\yb^*\|^2}{{\cal D}_{X^m}^2\delta^2}\right)\right\},
		\eeq
respectively. In particular, if $\epsilon$ and $\delta$ satisfy
\beq \label{tolerance_delta}
\tfrac{\epsilon}{\delta} \le \tfrac{{\cal D}_{X^m}^2}{{\cal D}_{X^m}+\|\yb^*\|} ,
\eeq
then the previous two complexity bounds in \eqnok{complexity_ns}, respectively, reduce to
\beq \label{simp_complexity_ns}
{\cal O} \left\{\tfrac{\|\Lb\| {\cal D}_{X^m}^2}{\epsilon}\right\} \ \mbox{and} \ \ {\cal O} \left\{\tfrac{m M^2 {\cal D}_{X^m}^2}{\epsilon^2} \right\}.
\eeq


Thirdly, it is interesting to compare DCS with
the centralized mirror descent method~\cite{nemyud:83} applied to \eqnok{eqn:orgprob}.
In the worst case, the Lipschitz constant of $f$ in  \eqnok{eqn:orgprob} can be bounded by
$M_f \le m M$, and each iteration of the method will incur $m$ subgradient evaluations.
Hence, the total number of subgradient evaluations performed by the mirror descent method
for finding an $\epsilon$-solution of  \eqnok{eqn:orgprob}, i.e.,  a point $\bar x \in X$
such that $f(\bar x) - f^* \le \epsilon$, can be bounded by
\beq \label{complexity_ns_MD}
{\cal O} \left\{ \tfrac{m^3 M^2 {\cal D}_{X}^2}{\epsilon^2}\right\},
\eeq
where ${\cal D}_{X}^2$ characterizes the diameter of $X$, i.e., ${\cal D}_X^2 := \max_{x_1, x_2 \in X} V(x_1,x_2)$.
Noting that ${\cal D}_{X}^2 / {\cal D}_{X^m}^2 = {\cal O}(1 / m)$, and that  the second bound in \eqnok{simp_complexity_ns}
states only the number of subgradient evaluations for each agent in the DCS method,
we conclude that the total number of subgradient evaluations performed by DCS is comparable to
the classic mirror descent method as long as \eqnok{tolerance_delta} holds and hence not improvable in general.

\subsection{Boundedness of $\|\yb^*\|$}

In this subsection, we will provide a bound on the optimal dual multiplier $\yb^*$.
By doing so,
we show that the complexity of DCS algorithm (as well as the stochastic DCS algorithm
in Section 5) only depends on the parameters for the primal problem along with
the smallest singular value of $\Lb$ and the initial point $\yb^0$, even though these algorithms are
intrinsically primal-dual type methods.

%
\begin{theorem}\label{thm:bnd_y}
	Let $\xb^*$ be an optimal solution of \eqref{eqn:prob}. Then there exists an optimal dual multiplier $\yb^*$ for \eqref{eqn:saddle} s.t.
	\beq\label{eqn:bnd_y}
	\|\yb^*\|\le \tfrac{\sqrt{m}M}{\tilde\sigma_{min}(\Lb)},
	\eeq
	where $\tilde\sigma_{min}(\Lb)$ denotes the smallest nonzero singular value of $\Lb$.
\end{theorem}
\begin{proof}
	Since we only relax the linear constraints in problem \eqref{eqn:prob} to obtain the Lagrange dual problem \eqref{eqn:saddle},
	it follows from the strong Lagrange duality and the existence of $\xb^*$ to \eqnok{eqn:prob} that an optimal dual multiplier $\yb^*$ for problem \eqref{eqn:saddle} must exist.
	It is clear that
	\[
	\yb^*=\yb^*_{N} + \yb^*_{C}, 
	\]
	where $\yb^*_{N}$ and $\yb^*_{C}$ denote the projections of $\yb^*$ over the null space and the column space of $\Lb^T$, respectively.
	
	We consider two cases. Case 1) $\yb^*_{C}=\0b$. Since $\yb^*_{N}$ belongs to the null space of $\Lb^T$,
	$\Lb^T\yb^*=\Lb^T\yb^*_{N}=\0b$, which implies that for any $c\in \bbr$, $c\yb^*$ is also an optimal dual multiplier of \eqref{eqn:saddle}.
	Therefore, \eqref{eqn:bnd_y} clearly holds, because we can scale $\yb^*$ to an arbitrarily small vector.
	
	Case 2) $\yb^*_{C} \ne \0b$. Using the fact that $\Lb^T\yb^*=\Lb^T\yb^*_{C}$ and the definition of a saddle point of \eqref{eqn:saddle},
	we conclude that $\yb^*_{C}$ is also an optimal dual multiplier of \eqref{eqn:saddle}.
	Since $\yb^*_C$ is in the column space of $\Lb$, we have
	\[
	\|\Lb^T\yb^*_C\|^2
	=(\yb^*_C)^T\Lb\Lb^T\yb^*_C=(\yb^*_C)^T\mathbf{U}^T\mathbf{\Lambda}\mathbf{U}\yb^*_C
	\ge \tilde\lambda_{\min}(\Lb\Lb^T)\|\mathbf{U}\yb^*_C\|^2
	=\tilde\sigma_{\min}^2(\Lb)\|\yb^*_C\|^2,
	\]
	where $\mathbf{U}$ is an orthonormal matrix whose rows consist of the eigenvectors of $\Lb\Lb^T$, $\mathbf{\Lambda}$ is
	the diagonal matrix whose diagonal elements are the corresponding eigenvalues, $\tilde\lambda_{min}(\Lb\Lb^T)$ denotes
	the smallest nonzero eigenvalue of $\Lb\Lb^T$, and
$\tilde\sigma_{min}(\Lb)$ denotes the smallest nonzero singular value of $\Lb$.	
	In particular,
	\beq\label{bnd_y_sigma}
	\|\yb^*_{C}\|\le \tfrac{\|\Lb^T\yb^*_{C}\|}{\tilde\sigma_{\min}(\Lb)}.
	\eeq
	Moreover, if we denote the saddle point problem defined in \eqref{eqn:saddle} as follows:
	\[
	{\cal L}(\xb,\yb):=F(\xb)+\la \Lb\xb,\yb\ra.
	\]
	By the definition of a saddle point of \eqref{eqn:saddle}, we have	${\cal L}(\xb^*,\yb^*_{C})\le {\cal L}(\xb, \yb^*_{C})$, i.e.,
	\[
	F(\xb^*)-F(\xb)\le \la -\Lb^T\yb^*_{C},\xb-\xb^*\ra.
	\]
	Hence, from the definition of subgradients, we conclude that $-\Lb^T\yb^*_{C}\in \partial F(\xb^*)$, which together with the fact that $F(\cdot)$ is Lipschitz continuous implies that
	\[
	\|\Lb^T\yb^*_{C}\|=\|\tsum_{i=1}^mf_i'(x_i^*)\|\le \sqrt{m}M.
	\]
	Our result in \eqref{eqn:bnd_y} follows immediately from the above relation, \eqref{bnd_y_sigma} and the fact that $\yb^*_{C}$ is also an optimal dual multiplier of \eqref{eqn:saddle}.
\end{proof}

Observe that our bound for the dual multiplier $\yb^*$ in \eqref{eqn:bnd_y} contains only the primal information.
Given an initial dual multiplier $\yb^0$, this result can be used to provide
an upper bound on $\|\yb^0-\yb^*\|$ in Theorems \ref{exact_thm}-\ref{stoch_thm_s} throughout this paper.
Note also that we can assume $\yb^0 = 0$ to simplify these complexity bounds.

\subsection{Convergence of DCS on Strongly Convex Functions}
In this subsection, we assume that the objective functions $f_i$'s are strongly convex (i.e.,  $\mu > 0$ \eqref{eqn:nonsmooth}).
In order to take advantage of the strong convexity of the objective functions,
we assume that the prox-functions $V_i(\cdot,\cdot)$, $i=1,\dots,m$, (cf. \eqref{eqn:def_V}) are growing quadratically
with the \textit{quadratic growth constant} $\mathcal{C}$, i.e.,
there exists a constant $\mathcal{C}>0$ such that
\begin{align}\label{eqn:proxquad}
V_i(x_i,u_i) \le \tfrac{\mathcal{C}}{2}\|x_i-u_i\|_{X_i }^2, \quad \forall x_i, u_i \in X_i,\ i=1,\ldots,m.
\end{align}
By \eqref{eqn:def_V_i_s}, we must have $\mathcal{C}\ge 1$.

We next provide in Lemma \ref{lem:main_deter_sc}
an estimate on the gap function defined in \eqref{eqn:gap}
together with stepsize policies which work for the strongly convex case.
The proof of this lemma can be found in Section \ref{sec:conv}.

\begin{lemma}\label{lem:main_deter_sc}
Let the iterates $(\hat{\xb}^k,\yb^k)$, $k = 1, \ldots, N$ be generated by Algorithm~\ref{alg:DCS}
and $\hat{\zb}^N$ be defined as
$\hat{\zb}^N := \left(\tsum_{k=1}^N\theta_k\right)^{-1}\tsum_{k=1}^N\theta_k(\hat \xb^k,\yb^k)$.
Assume the objective $f_i$, $i = 1, \ldots, m$ are strongly convex functions, i.e., $\mu, M >0$.
Let the parameters $\{\a_k\}$, $\{\t_k\}$, $\{\eta_k\}$ and $\{\tau_k\}$ in Algorithm~\ref{alg:DCS} satisfy
\eqref{alpha_theta}-\eqref{eta_tau_theta} and
\begin{align}\label{theta_eta_ds}
\theta_k\eta_k \le \theta_{k-1}(\mu/\mathcal{C}+\eta_{k-1}), \quad k = 2, \ldots, N.
\end{align}
Let the parameters $\{\lambda_t\}$ and $\{\beta_t\}$ in the CS procedure of Algorithm~\ref{alg:DCS} be set to
\begin{align}\label{eqn:csparam_sc}
\lambda_t = t, \quad \beta^{(k)}_t = \tfrac{(t+1)\mu}{2\eta_k\mathcal{C}} + \tfrac{t-1}{2}, \quad \forall t \ge1.
\end{align}
Then, we have for all $\zb \in X^m\times \bbr^{md}$
    \begin{align}\label{eqn:Q02}
    Q(\hat{\zb}^N;\zb)
    &\le \left(\tsum_{k=1}^N \theta_k\right)^{-1}
    \Bigg[\theta_1\eta_1\mathbf{V}(\xb^0,\xb)+\tfrac{\theta_1\tau_1}{2}\|\yb^0\|^2 + \la \hat\vb,\yb\ra
    +\tsum_{k=1}^N\tsum_{t=1}^{T_k}\tfrac{2mM^2\theta_k}{T_k(T_k+1)}\tfrac{t} {(t+1)\mu/\mathcal{C}+(t-1)\eta_k}\Bigg],
    \end{align}
where $\hat\vb := \theta_N \Lb (\hat \xb^N - \xb^{N-1}) + \theta_1\tau_1 (\yb^N - \yb^0)$
and $Q$ is defined in \eqref{eqn:gap}.
    Furthermore, for any saddle point $(\xb^*,\yb^*)$ of \eqref{eqn:saddle}, we have
    \begin{align}\label{eqn:xNbnd}
    \tfrac{\theta_N}{2}\left(1-\tfrac{\|\Lb\|^2}{\eta_N\tau_N}\right)&\max\{\eta_N\|\hat\xb^N-\xb^{N-1}\|^2,\tau_N\|\yb^*-\yb^N\|^2\}\\\nn
    &~\le \theta_1\eta_1\mathbf{V}(\xb^0, \xb^*)+\tfrac{\theta_1\tau_1}{2}\|\yb^*-\yb^0\|^2
    +\tsum_{k=1}^N\tsum_{t=1}^{T_k}\tfrac{2mM^2\theta_k}{T_k(T_k+1)}\tfrac{t} {(t+1)\mu/\mathcal{C}+(t-1)\eta_k}.
    \end{align}
\end{lemma}

In the following theorem, we provide a specific selection of $\{\alpha_k\}$, $\{\theta_k\}$, $\{\eta_k\}$, $\{\tau_k\}$ and $\{T_k\}$  satisfying \eqref{alpha_theta}-\eqref{eta_tau_theta} and \eqref{theta_eta_ds}. Also, by using Lemma \ref{lem:main_deter_sc} and Proposition \ref{prop:approx}, we establish the complexity of the DCS method for computing an $(\epsilon,\delta)$-solution of problem \eqref{eqn:prob} when the objective functions are strongly convex. The choice of variable stepsizes rather than using constant stepsizes
will accelerate its convergence rate.
\begin{theorem}\label{deter_thm_s}
	Let $\xb^*$ be an optimal solution of \eqnok{eqn:prob}, the parameters $\{\lambda_t\}$ and $\{\beta_t\}$ in the CS procedure of Algorithm~\ref{alg:DCS} be set to
	\eqref{eqn:csparam_sc} and suppose that $\{\alpha_k\}$, $\{\theta_k\}$,
$\{\eta_k\}$, $\{\tau_k\}$ and $\{T_k\}$ are set to
	\beq\label{deter_s_para}
    \alpha_k=\tfrac{k}{k+1},~ \theta_k = k+1,~\eta_k=\tfrac{k\mu}{2\mathcal{C}},~\tau_k=\tfrac{4\|\Lb\|^2\mathcal{C}}{(k+1)\mu},
	\mbox{ and }T_k=\left\lceil \sqrt{\tfrac{2m}{\tilde D}}\tfrac{\mathcal{C}MN}{\mu}\max\left\{\sqrt{\tfrac{2m}{\tilde D}}\tfrac{4\mathcal{C}M}{\mu}, 1\right\}\right\rceil,
	\eeq
	$\forall k=1, \ldots, N$, for some $\tilde D>0$. Then, for any $N \ge 2$, we have
\begin{align}\label{eqn:deterobjs}
    F(\hat\xb^N)-F(\xb^*)
    &\le \tfrac{2}{N(N+3)}
    \left[\tfrac{\mu}{\mathcal{C}}\mathbf{V}(\xb^0,\xb^*)+\tfrac{2\|\Lb\|^2\mathcal{C}}{\mu}\|\yb^0\|^2 + \tfrac{2\mu\tilde D}{\mathcal{C}}\right],
\end{align}
and
\begin{align}\label{eqn:deterfeass}
\|\Lb\hat\xb^N\|
\le &~ \tfrac{8\|\Lb\|}{N(N+3)} \left[3\sqrt{2\tilde D+ \mathbf{V}(\xb^0,\xb^*)}+\tfrac{7\|\Lb\|\mathcal{C}}{\mu}\|\yb^*-\yb^0\|\right],
\end{align}
where
$\hat\xb^N = \tfrac{2}{N(N+3)}\tsum_{k=1}^N(k+1)\hat\xb^k$.
\end{theorem}
\begin{proof}
	It is easy to check that \eqnok{deter_s_para} satisfies conditions \eqnok{alpha_theta}-\eqnok{eta_tau_theta}
and \eqref{theta_eta_ds}. Moreover, we have
	\begin{align*}
	\tsum_{k=1}^N\tsum_{t=1}^{T_k}\tfrac{2mM^2\theta_k}{T_k(T_k+1)}\tfrac{t} {(t+1)\mu/\mathcal{C}+(t-1)\eta_k}
	&=\tsum_{k=1}^N\tfrac{2mM^2\theta_k\mathcal{C}}{T_k(T_k+1)\mu} \tsum_{t=1}^{T_k}\tfrac{2t}{2(t+1)+(t-1)k}\\
	&\le \tsum_{k=1}^N\tfrac{2mM^2\theta_k\mathcal{C}}{T_k(T_k+1)\mu}\left(\tfrac{1}{2}+\tsum_{t=2}^{T_k}\tfrac{2t}{(t-1)(k+1)}\right)\\
	&\le \tsum_{k=1}^N\tfrac{mM^2\mathcal{C}(k+1)}{T_k(T_k+1)\mu}+\tsum_{k=1}^N\tfrac{8mM^2\mathcal{C}(T_k-1)}{T_k(T_k+1)\mu}
	\le \tfrac{2\mu\tilde D}{\mathcal{C}}.
	\end{align*}
	Therefore, by plugging in these values to \eqnok{eqn:Q02}, we have
\begin{align}\label{eqn:qdeters}
    Q(\hat{\zb}^N;\xb^*,\yb)
    &\le \tfrac{2}{N(N+3)}
    \left[\tfrac{\mu}{\mathcal{C}} \mathbf{V}(\xb^0,\xb^*)+\tfrac{2\|\Lb\|^2\mathcal{C}}{\mu}\|\yb^0\|^2 + \tfrac{2\mu\tilde D}{\mathcal{C}}\right]
    +\tfrac{2}{N(N+3)} \la \hat\vb,\yb\ra.
\end{align}
Furthermore, from \eqref{eqn:xNbnd}, we have for $N\ge2$
\begin{align}\label{eqn:xbnd}
\|\hat\xb^N-\xb^{N-1}\|^2 &\le \tfrac{8\mathcal{C}}{\mu(N+1)(N-1)}\left[\tfrac{\mu}{\mathcal{C}} \mathbf{V}(\xb^0,\xb^*) + \tfrac{2\|\Lb\|^2\mathcal{C}}{\mu}\|\yb^0 - \yb^*\|^2+\tfrac{2\mu\tilde D}{\mathcal{C}}\right],\\\nn
\|\yb^*-\yb^N\|^2 &\le
\tfrac{N\mu}{(N-1)\|\Lb\|^2\mathcal{C}}\left[\tfrac{\mu}{\mathcal{C}} \mathbf{V}(\xb^0,\xb^*) + \tfrac{2\|\Lb\|^2\mathcal{C}}{\mu}\|\yb^0 - \yb^*\|^2+\tfrac{2\mu\tilde D}{\mathcal{C}}\right].
\end{align}
Let $\vb^N := \tfrac{2}{N(N+3)}\hat\vb$, then by using \eqref{eqn:xbnd}, we have for $N\ge2$
\begin{align*}
\|\vb^N\| \le &~\tfrac{2}{N(N+3)} \left[(N+1)\|\Lb\|\|\hat\xb^N-\xb^{N-1}\|+\tfrac{4\|\Lb\|^2\mathcal{C}}{\mu}\|\yb^N-\yb^*\|+\tfrac{4\|\Lb\|^2\mathcal{C}}{\mu}\|\yb^*-\yb^0\|\right]\\\nn
\le &~ \tfrac{8\|\Lb\|}{N(N+3)} \left[3\sqrt{2\tilde D+ \mathbf{V}(\xb^0,\xb^*) + \tfrac{2\|\Lb\|^2\mathcal{C}^2}{\mu^2}\|\yb^0 - \yb^*\|^2}+\tfrac{\|\Lb\|\mathcal{C}}{\mu}\|\yb^*-\yb^0\|\right]\\\nn
\le &~ \tfrac{8\|\Lb\|}{N(N+3)} \left[3\sqrt{2\tilde D+ \mathbf{V}(\xb^0,\xb^*)}+\tfrac{7\|\Lb\|\mathcal{C}}{\mu}\|\yb^*-\yb^0\|\right].
\end{align*}
From \eqref{eqn:qdeters}, we further have
\begin{align*}
    g(\hat\vb^N,\hat{\zb}^N)
    &\le \tfrac{2}{N(N+3)}
    \left[\tfrac{\mu}{\mathcal{C}} \mathbf{V}(\xb^0,\xb^*)+\tfrac{2\|\Lb\|^2\mathcal{C}}{\mu}\|\yb^0\|^2 + \tfrac{2\mu\tilde D}{\mathcal{C}}\right].
\end{align*}
Applying Proposition \ref{prop:approx} to the above two inequalities, the results in \eqref{eqn:deterobjs} and \eqref{eqn:deterfeass} follow immediately.
\end{proof}

			We now make some remarks about the results obtained in Theorem~\ref{deter_thm_s}. Firstly,
			similar to the general convex case, the best choice for $\tilde D$ (cf. \eqref{deter_s_para}) would be $\mathbf{V}(\xb^0,\xb^*)$ so that the first and the third terms in \eqref{eqn:qdeters} are about the same order. If there exists an estimate ${\cal D}_{X^m}>0$ satisfying \eqref{V_bndness}, we can set $\tilde D={\cal D}_{X^m}^2$.
		
			Secondly, the complexity of the DCS method for solving strongly convex problems follows from \eqref{eqn:deterobjs} and \eqref{eqn:deterfeass}. For simplicity, let us assume that $X$ is bounded, $\tilde D={\cal D}_{X^m}^2$ and $\yb^0=\0b$. We can see that the total number of inter-node communication rounds and intra-node subgradient evaluations performed by each agent for finding an $(\epsilon, \delta)$-solution of \eqref{eqn:prob} can be bounded by
		\beq\label{complexity_sc}
		{\cal O}\left\{\max\left(\sqrt{\tfrac{\mu{\cal D}_{X^m}^2}{\mathcal{C}\epsilon}},\sqrt{\tfrac{\|\Lb\|}{\delta}\left({\cal D}_{X^m}+\tfrac{\mathcal{C}\|\Lb\|\|\yb^*\|}{\mu}\right)}\right)\right\} \ \ \mbox{and} \ \
		{\cal O}\left\{\tfrac{mM^2\mathcal{C}}{\mu}\max\left(\tfrac{1}{\epsilon}, \tfrac{\|\Lb\|\mathcal{C}}{\mu\delta}\left(\tfrac{1}{{\cal D}_{X^m}}+\tfrac{\mathcal{C}\|\Lb\|\|\yb^*\|}{{\cal D}_{X^m}^2\mu}\right)\right)\right\},
		\eeq
respectively. In particular, if $\epsilon$ and $\delta$ satisfy
\beq\label{tolerance_delta_s}
\tfrac{\epsilon}{\delta}\le \tfrac{\mu^2{\cal D}_{X^m}^2}{\|\Lb\|\mathcal{C}(\mu{\cal D}_{X^m}+\mathcal{C}\|\Lb\|\|\yb^*\|)},
\eeq
then the complexity bounds in \eqref{complexity_sc}, respectively, reduce to
\beq\label{simp_complexity_sc}
{\cal O}\left\{\sqrt{\tfrac{\mu{\cal D}_{X^m}^2}{\mathcal{C}\epsilon}}\right\} \ \
\mbox{and} \ \
{\cal O}\left\{\tfrac{mM^2\mathcal{C}}{\mu\epsilon}\right\}.
\eeq

Thirdly, we compare DCS method with the centralized mirror descent method~\cite{nemyud:83} applied to \eqnok{eqn:orgprob}.
In the worst case, the Lipschitz constant and strongly convex modulus of $f$ in \eqnok{eqn:orgprob} can be bounded by
$M_f \le m M$, and $\mu_f\ge m\mu$, respectively, and each iteration of the method will incur $m$ subgradient evaluations. Therefore, the total number of subgradient evaluations performed by the mirror descent method for finding an $\epsilon$-solution of \eqref{eqn:orgprob}, i.e., a point $\bar x\in X$ such that $f(\bar x)-f^*\le \epsilon$, can be bounded by
\beq\label{complexity_s_MD}
{\cal O}\left\{\tfrac{m^2M^2\mathcal{C}}{\mu\epsilon}\right\}.
\eeq
Observed that the second bound in \eqref{simp_complexity_sc}
	states only the number of subgradient evaluations for each agent in the DCS method,
	we conclude that the total number of subgradient evaluations performed by DCS is comparable to
	the classic mirror descent method as long as \eqnok{tolerance_delta_s} holds and hence not improvable in general
	for the nonsmooth strongly convex case.


\setcounter{equation}{0}
\section{Stochastic Decentralized Communication Sliding}
\label{sec:stochastic}

In this section, we consider the stochastic case
where only the noisy subgradient information of the functions $f_i$,  $i = 1, \ldots, m$,
is available or easier to compute.
This situation happens when the function $f_i$'s are given either in the form of expectation
or as the summation of lots of components.
This setting has attracted considerable interest in recent decades for its applications in a
broad spectrum of disciplines including machine learning, signal processing, and operations research.
We present a stochastic communication sliding method, namely the stochastic decentralized communication sliding (SDCS) method, and show that the similar complexity bounds as in Section \ref{sec:deterministic}
can still be obtained in expectation or with high probability.

\subsection{The SDCS Algorithm}
The first-order information of the function $f_i$, $i = 1, \ldots, m$,
can be accessed by a stochastic oracle (SO), which,  given a point $u^t \in X$,
outputs a vector $G_i(u^t,\xi_i^t)$ such that
\begin{align}\label{assume:unbiased}
\bbe[G_i(u^t,\xi_i^t)]=f_i'(u^t) \in \partial f_i(u^t),
\end{align}
\begin{align}\label{assume:sm_bounded}
\bbe[\|G_i(u^t,\xi_i^t)-f_i'(u^t)\|_*^2]\le \sigma^2,
\end{align}
where $\xi_i^t$ is a random vector which models a source of uncertainty and is independent of the search point $u^t$, and the distribution $\mathbb{P}(\xi_i)$ is not known in advance.
We call $G_i(u^t,\xi_i^t)$ a \textit{stochastic subgradient} of $f_i$ at $u^t$.

The SDCS method can be obtained by simply replacing the exact subgradients
in the CS procedure of Algorithm~\ref{alg:DCS} with the stochastic subgradients obtained from SO.
This difference is described in Algorithm~\ref{alg:SDCS}.

\begin{algorithm}
\caption{SDCS}
\label{alg:SDCS}
\begin{algorithmic}
\STATE  The projection step \eqref{eqn:inner0}-\eqref{eqn:inner1} in the CS procedure
of Algorithm~\ref{alg:DCS} is replaced by
\begin{align}
h^{t-1} = &~  H(u^{t-1},\xi^{t-1}),\label{eqn:inner0-s}\\
u^t = &~ \argmin_{u \in U}\left[\la w + h^{t-1}, u \ra +\eta V(x,u) + \eta\b_tV(u^{t-1},u)\right],\label{eqn:inner1-s}
\end{align}
\STATE where $H(u^{t-1},\xi^{t-1})$ is a stochastic subgradient of $\phi$ at $u^{t-1}$.
\end{algorithmic}
\end{algorithm}

We add a few remarks about the SDCS algorithm.
Firstly, as in DCS,
no additional communications of the dual variables are required
when the subgradient projection \eqref{eqn:inner1-s} is performed for $T_k$ times in the inner loop.
This is because the same $w_i^k$ has been used throughout the $T_k$ iterations of the Stochastic CS procedure.
Secondly, the problem will reduce to the deterministic case if there is no stochastic noise
associated with the SO, i.e., when $\sigma = 0$ in \eqref{assume:sm_bounded}.
Therefore, in Section \ref{sec:conv}, we investigate the convergence analysis for the stochastic case first and then simplify the analysis for the deterministic case by setting $\sigma = 0$.

\subsection{Convergence of SDCS on General Convex Functions\label{sec:sgc}}

We now establish the main convergence properties of the SDCS algorithm.
More specifically, we provide in Lemma \ref{lem:main_stoch}
an estimate on the gap function defined in \eqref{eqn:gap} together with
stepsize policies which work for the general convex case with $\mu = 0$ (cf. \eqref{eqn:nonsmooth}).
The proof of this lemma can be found in Section \ref{sec:conv}.
\begin{lemma}\label{lem:main_stoch}
Let the iterates $(\hat{\xb}^k,\yb^k)$ for $k = 1, \ldots, N$ be generated by Algorithm~\ref{alg:SDCS} and $\hat{\zb}^N$ be defined as $\hat{\zb}^N := \left(\tsum_{k=1}^N \t_k\right)^{-1}\tsum_{k=1}^N \t_k (\hat{\xb}^k,\yb^k)$.
Assume the objective $f_i$, $i = 1, \ldots, m$, are general nonsmooth convex functions, i.e., $\mu = 0$ and $M >0$.
Let the parameters $\{\a_k\}$, $\{\t_k\}$, $\{\eta_k\}$, $\{\tau_k\}$ and $\{T_k\}$ in Algorithm~\ref{alg:SDCS} satisfy
\eqref{alpha_theta}-\eqref{eta_tau_theta} and \eqref{theta_eta_d}.
Let the parameters $\{\lambda_t\}$ and $\{\beta_t\}$ in the CS procedure of Algorithm~\ref{alg:SDCS} be set as \eqref{eqn:csparam}.
Then, for all $\zb \in X^m\times \bbr^{md}$,
\begin{align}\label{Q_pert_bnd}
Q(\hat{\zb}^N;\zb) \le&~\left(\tsum_{k=1}^N\theta_k\right)^{-1}\Bigg\{\tfrac{(T_1+1)(T_1+2)\theta_1\eta_1}{T_1(T_1+3)}\mathbf{V}(\xb^0,\xb)+\tfrac{\theta_1\tau_1}{2}\|\yb^0\|^2+\la\hat \vb,\yb\ra\\\nn
&~~~+ \tsum_{k=1}^N\tsum_{t=1}^{T_k}\tsum_{i=1}^m\tfrac{2\theta_k}{T_k(T_k+3)}\left[(t+1)\langle\delta_i^{t-1,k}, x_i-u_i^{t-1}\rangle
+\tfrac{4(M^2+\|\delta_i^{t-1,k}\|_*^2)}{\eta_k}\right]\Bigg\},
\end{align}
where
$\hat \vb := \theta_N \Lb (\hat \xb^N -\xb^{N-1}) + \theta_1\tau_1 (\yb^N-\yb^0)$
and $Q$ is defined in \eqref{eqn:gap}. Furthermore, for any saddle point $(\xb^*,\yb^*)$ of \eqref{eqn:saddle}, we have
\begin{align}\label{xNs_bnd}
\tfrac{\theta_N}{2}\left(1-\tfrac{\|\Lb\|^2}{\eta_N\tau_N}\right)&\max\{\eta_N\|\hat\xb^N-\xb^{N-1}\|^2,\tau_N\|\yb^*-\yb^N\|^2\}\\\nn
&~\le \tfrac{(T_1+1)(T_1+2)\theta_1\eta_1}{T_1(T_1+3)}\mathbf{V}(\xb^0,\xb^*)+\tfrac{\theta_1\tau_1}{2}\|\yb^*-\yb^0\|^2\\\nn
&~~~+ \tsum_{k=1}^N\tsum_{t=1}^{T_k}\tsum_{i=1}^m\tfrac{2\theta_k}{T_k(T_k+3)}\left[(t+1)\langle\delta_i^{t-1,k}, x_i^*-u_i^{t-1}\rangle
+\tfrac{4(M^2+\|\delta_i^{t-1,k}\|_*^2)}{\eta_k}\right].
\end{align}
\end{lemma}

In the following theorem, we provide a specific selection of $\{\alpha_k\}$, $\{\theta_k\}$, $\{\eta_k\}$, $\{\tau_k\}$ and $\{T_k\}$  satisfying \eqref{alpha_theta}-\eqref{eta_tau_theta} and \eqref{theta_eta_d}. Also, by using Lemma \ref{lem:main_stoch} and Proposition \ref{prop:approx}, we establish the complexity of the SDCS method for computing an $(\epsilon,\delta)$-solution of problem \eqref{eqn:prob} in expectation when the objective functions are general convex.

\begin{theorem}\label{stoch_thm_ns}
	Let $\xb^*$ be an optimal solution of \eqnok{eqn:prob}, the parameters $\{\lambda_t\}$ and $\{\beta_t\}$ in the CS procedure of Algorithm~\ref{alg:SDCS} be set as \eqref{eqn:csparam}, and suppose that  $\{\alpha_k\}$, $\{\theta_k\}$,
$\{\eta_k\}$, $\{\tau_k\}$ and $\{T_k\}$ are set to
	\beq\label{stoch_ns_para}
    \alpha_k=\theta_k = 1,~\eta_k=2\|\Lb\|,~\tau_k=\|\Lb\|,\mbox{ and }T_k=\left\lceil \tfrac{m (M^2+\sigma^2)N}{\|\Lb\|^2\tilde D}\right\rceil, \quad \forall k=1, \ldots, N,
	\eeq
	for some $\tilde D>0$. Then, under Assumptions \eqref{assume:unbiased} and \eqref{assume:sm_bounded}, we have for any $N \ge 1$
	\begin{align}\label{bnd_func}
	\bbe[F(\hat \xb^k)-F(\xb^*)]&\le \tfrac{\|\Lb\|}{N}\left[3\mathbf{V}(\xb^0,\xb^*)+\tfrac{1}{2}\|\yb^0\|^2+4\tilde D\right],
	\end{align}
	and
	\begin{align}\label{bnd_infeas}
	\bbe[\|\Lb\hat \xb^N\|]&\le \tfrac{\|\Lb\|}{N}\left[3\sqrt{6\mathbf{V}(\xb^0,\xb^*)+8\tilde D}+4\|\yb^*-\yb^0\|\right].
	\end{align}
	where $\hat \xb^N=\tfrac{1}{N}\tsum_{k=1}^N\hat \xb^k$.
\end{theorem}
\begin{proof}
	It is easy to check that \eqnok{stoch_ns_para} satisfies conditions \eqnok{alpha_theta}-\eqnok{eta_tau_theta} and \eqref{theta_eta_d}. Moreover, by \eqref{eqn:gapp}, we can obtain
	\begin{align}\label{g_bnd}
	g(\hat \vb^N,\hat \zb^N)
	=&~\max_{\yb}Q(\hat \zb^N;\xb^*,\yb)-\left(\tsum_{k=1}^N\theta_k\right)^{-1}\la \hat \vb,\yb\ra \\\nn
	\le&~\left(\tsum_{k=1}^N\theta_k\right)^{-1}\Bigg\{\tfrac{(T_1+1)(T_1+2)\theta_1\eta_1}{T_1(T_1+3)}\mathbf{V}(\xb^0,\xb^*)+\tfrac{\theta_1\tau_1}{2}\|\yb^0\|^2\\\nn
	&~~~+ \tsum_{k=1}^N\tsum_{t=1}^{T_k}\tsum_{i=1}^m\tfrac{2\theta_k}{T_k(T_k+3)}\left[(t+1)\langle\delta_i^{t-1,k}, x_i^*-u_i^{t-1}\rangle
	+\tfrac{4(M^2+\|\delta_i^{t-1,k}\|_*^2)}{\eta_k}\right]\Bigg\},
	\end{align}
	where $\vb^N=\left(\tsum_{k=1}^N\theta_k\right)^{-1}\hat \vb$.
	Particularly,
	from Assumption~\eqref{assume:unbiased} and \eqref{assume:sm_bounded},
	\[
	\bbe[\delta_i^{t-1,k}]=0,\ \
	\bbe[\|\delta_i^{t-1,k}\|_*^2]\le \sigma^2,\ \forall i\in \{1,\dots,m\},\  t\ge1,\ k\ge1,
	\]
	and from \eqref{stoch_ns_para}
	\[
	\tfrac{(T_1+1)(T_1+2)}{T_1(T_1+3)}=1+\tfrac{2}{T_1^2+3T_1}\le \tfrac{3}{2}.
	\]
	Therefore, by taking expectation over both sides of \eqref{g_bnd} and plugging in these values into \eqref{g_bnd}, we have
	\begin{align}\label{g_bnd_exp}
	\bbe[g(\hat \vb^N,\hat \zb^N)]
	\le&~\left(\tsum_{k=1}^N\theta_k\right)^{-1}\Bigg\{\tfrac{(T_1+1)(T_1+2)\theta_1\eta_1}{T_1(T_1+3)}\mathbf{V}(\xb^0,\xb)+\tfrac{\theta_1\tau_1}{2}\|\yb^0\|^2+\tsum_{k=1}^N\tfrac{8m(M^2+\sigma^2)\theta_k}{(T_k+3)\eta_k}
	\Bigg\}\\\nn
	\le&~ ~\tfrac{\|\Lb\|}{N}\left[3\mathbf{V}(\xb^0,\xb^*)+\tfrac{1}{2}\|\yb^0\|^2+4\tilde D\right],
	\end{align}
	with
	\[
	\bbe[\|\hat \vb^N\|]=\tfrac{1}{N}\bbe[\|\hat \vb\|]\le \tfrac{\|\Lb\|}{N}\bbe\left[\|\hat \xb^N-\xb^{N-1}\|+\|\yb^N-\yb^*\|+\|\yb^*-\yb^0\|\right].
	\]
	Note that from \eqref{xNs_bnd} and Jensen's inequality, we have
	\begin{align*}
	(\bbe[\|\hat \xb^N-\xb^{N-1}])^2&\le \bbe[\|\hat \xb^N-\xb^{N-1}\|^2]\le 6\mathbf{V}(\xb^0,\xb^*)+\|\yb^*-\yb^0\|+8\tilde D,\\
    (\bbe[\|\yb^*-\yb^N\|])^2&\le \bbe[\|\yb^*-\yb^N\|^2]\le 12\mathbf{V}(\xb^0,\xb^*)+2\|\yb^*-\yb^0\|+16\tilde D.
	\end{align*}	
	Hence,
	\begin{align*}
	\bbe[\|\hat \vb^N\|]\le \tfrac{\|\Lb\|}{N}\left[3\sqrt{6\mathbf{V}(\xb^0,\xb^*)+8\tilde D}+4\|\yb^*-\yb^0\|\right].
	\end{align*}
	Applying Proposition~\ref{prop:approx} to the above inequality and \eqref{g_bnd_exp}, the results in \eqref{bnd_func} and \eqref{bnd_infeas} follow immediately.
\end{proof}

	We now make some observations about the results obtained in Theorem~\ref{stoch_thm_ns}. Firstly, one can choose any $\tilde D>0$ (e.g., $\tilde D=1$) in \eqref{stoch_ns_para}, however, the best selection of $\tilde D$ would be $\mathbf{V}(\xb^0,\xb^*)$ so that the first and third terms in \eqref{g_bnd_exp} are about the same order. In practice, if there exists an estimate ${\cal D}_{X^m}>0$ satisfying \eqref{V_bndness},
	we can set $\tilde D={\cal D}_{X^m}^2$.

	Secondly, the complexity of SDCS method immediately follows from \eqref{bnd_func} and \eqref{bnd_infeas}. Under the above assumption, with $\tilde D={\cal D}^2_{X^m}$ and $\yb^0=\0b$, we can see that the total number of inter-node communication rounds and intra-node subgradient evaluations required by each agent for finding a stochastic $(\epsilon,\delta)$-solution of \eqref{eqn:prob} can be bounded by
	\beq\label{complexity_stoch_ns}
	 {\cal O}\left\{\|\Lb\|\max\left(\tfrac{{\cal D}_{X^m}^2}{\epsilon},\tfrac{{\cal D}_{X^m}+\|\yb^*\|}{\delta}\right)\right\}\ \
	\mbox{and} \ \
	{\cal O}\left\{m(M^2+\sigma^2)\max\left(\tfrac{{\cal D}_{X^m}^2}{\epsilon^2},\tfrac{{\cal D}_{X^m}^2+\|\yb^*\|^2}{{\cal D}_{X^m}^2\delta^2}\right)\right\},
	\eeq
	respectively.
	In particular, if $\epsilon$ and $\delta$ satisfy \eqref{tolerance_delta}, the above complexity bounds, respectively, reduce to
	\beq\label{simp_complexity_stoch_ns}
	{\cal O}\left\{\tfrac{\|\Lb\|{\cal D}_{X^m}^2}{\epsilon}\right\}\ \
	\mbox{and} \ \
	{\cal O}\left\{\tfrac{m(M^2+\sigma^2){\cal D}_{X^m}^2}{\epsilon^2}\right\}.
	\eeq
	In particular, we can show that the total number stochastic subgradients
	that SDCS requires is comparable to the mirror-descent stochastic approximation in \cite{NJLS09-1}.
	This implies that the sample complexity for decentralized stochastic
	optimization are still optimal (as the centralized one), even
	after we skip many communication rounds.

\subsection{Convergence of SDCS on Strongly Convex Functions\label{sec:ssc}}

We now provide in Lemma \ref{lem:main_stoch_sc}
an estimate on the gap function defined in \eqref{eqn:gap}
together with stepsize policies which work for the strongly convex case with $\mu > 0$ (cf. \eqref{eqn:nonsmooth}).
The proof of this lemma can be found in Section \ref{sec:conv}.

Note that throughout this subsection, we assume that the prox-functions $V_i(\cdot,\cdot)$, $i=1,\dots,m$, (cf. \eqref{eqn:def_V})
are growing quadratically with the quadratic growth constant $\mathcal{C}$, i.e., \eqref{eqn:proxquad} holds.

\begin{lemma}\label{lem:main_stoch_sc}
Let the iterates $(\hat{\xb}^k,\yb^k)$, $k = 1, \ldots, N$ be generated by Algorithm~\ref{alg:SDCS} and $\hat{\zb}^N$ be defined as $\hat{\zb}^N := \left(\tsum_{k=1}^N \t_k\right)^{-1}\tsum_{k=1}^N \t_k (\hat{\xb}^k,\yb^k)$.
Assume the objective $f_i$, $i = 1, \ldots, m$ are strongly convex functions, i.e., $\mu, M >0$.
Let the parameters $\{\a_k\}$, $\{\t_k\}$, $\{\eta_k\}$ and $\{\tau_k\}$ in Algorithm~\ref{alg:SDCS} satisfy
\eqref{alpha_theta}-\eqref{eta_tau_theta} and \eqref{theta_eta_ds}.
Let the parameters $\{\lambda_t\}$ and $\{\beta_t\}$ in the CS procedure of Algorithm~\ref{alg:SDCS} be set as \eqref{eqn:csparam_sc}.
Then, for all $\zb \in X^m\times \bbr^{md}$,
	\begin{align}\label{Q_expectation2}
	Q(\hat\zb^N;\zb)
	&~\le \left(\tsum_{k=1}^N\theta_k\right)^{-1}\Bigg\{\theta_1\eta_1\mathbf{V}(\xb^0,\xb)+\tfrac{\theta_1\tau_1}{2}\|\yb^0\|^2+\la\hat\vb,\yb\ra\\\nn &~~~+\tsum_{k=1}^N\tsum_{t=1}^{T_k}\tsum_{i=1}^m\tfrac{2\theta_k}{T_k(T_k+1)}\left[t\langle\delta_i^{t-1,k}, x_i-u_i^{t-1}\rangle
	+ \tfrac{2t(M^2+\|\delta_i^{t-1,k}\|_*^2)}{(t+1)\mu/\mathcal{C}+(t-1)\eta_k}\right]\Bigg\},
	\end{align}
where $\hat\vb := \theta_N \Lb (\hat \xb^N - \xb^{N-1}) + \theta_1\tau_1 (\yb^N - \yb^0)$
and $Q$ is defined in \eqref{eqn:gap}.
    Furthermore, for any saddle point $(\xb^*,\yb^*)$ of \eqref{eqn:saddle}, we have
    \begin{align}\label{eqn:xNbndstoch}
    \tfrac{\theta_N}{2}\left(1-\tfrac{\|\Lb\|^2}{\eta_N\tau_N }\right)&\max\{\eta_N \|\hat\xb^N-\xb^{N-1}\|^2,\tau_N\|\yb^*-\yb^N\|^2\}\\\nn
    &~\le \theta_1\eta_1\mathbf{V}(\xb^0, \xb^*)+\tfrac{\theta_1\tau_1}{2}\|\yb^*-\yb^0\|^2\\\nn
    &~~~+\tsum_{k=1}^N\tsum_{t=1}^{T_k}\tsum_{i=1}^m\tfrac{2\theta_k}{T_k(T_k+1)}\left[t\langle\delta_i^{t-1,k}, x_i^*-u_i^{t-1}\rangle
    + \tfrac{2t(M^2+\|\delta_i^{t-1,k}\|_*^2)}{(t+1)\mu/\mathcal{C}+(t-1)\eta_k}\right].
    \end{align}
\end{lemma}

In the following theorem, we provide a specific selection of $\{\alpha_k\}$, $\{\theta_k\}$, $\{\eta_k\}$, $\{\tau_k\}$ and $\{T_k\}$  satisfying \eqref{alpha_theta}-\eqref{eta_tau_theta} and \eqref{theta_eta_d}. Also, by using Lemma \ref{lem:main_stoch_sc} and Proposition \ref{prop:approx}, we establish the complexity of the SDCS method for computing an $(\epsilon,\delta)$-solution of problem \eqref{eqn:prob} in expectation when the objective functions are strongly convex. Similar to the deterministic case, we choose variable stepsizes rather than constant stepsizes.

\begin{theorem}\label{stoch_thm_s}
	Let $\xb^*$ be an optimal solution of \eqnok{eqn:prob}, the parameters $\{\lambda_t\}$ and $\{\beta_t\}$ in the CS procedure of Algorithm~\ref{alg:SDCS} be set as \eqref{eqn:csparam_sc}, and suppose that  $\{\alpha_k\}$, $\{\theta_k\}$,
$\{\eta_k\}$, $\{\tau_k\}$ and $\{T_k\}$ are set to
	\begin{align}\label{stoch_s_para}
    &\alpha_k=\tfrac{k}{k+1},~\theta_k = k+1,~\eta_k=\tfrac{k\mu}{2\mathcal{C}},~\tau_k=\tfrac{4\|\Lb\|^2\mathcal{C}}{(k+1)\mu},\mbox{ and }\\
	&\qquad T_k=\left\lceil \sqrt{\tfrac{m(M^2+\sigma^2)}{\tilde D}}\tfrac{2N\mathcal{C}}{\mu}\max\left\{\sqrt{\tfrac{m(M^2+\sigma^2)}{\tilde D}}\tfrac{8\mathcal{C}}{\mu}, 1\right\}\right\rceil, \quad \forall k=1, \ldots, N,\nonumber
	\end{align}
	for some $\tilde D>0$. Then, under Assumptions \eqref{assume:unbiased} and \eqref{assume:sm_bounded}, we have for any $N \ge 2$
	\beq\label{bnd_func_s}
	\bbe[F(\bar \xb^N)-F(\xb^*)\le \tfrac{2}{N(N+3)}\left[\tfrac{\mu}{\mathcal{C}} \mathbf{V}(\xb^0,\xb^*)+\tfrac{2\|\Lb\|^2\mathcal{C}}{\mu}\|\yb^0\|^2+\tfrac{2\mu\tilde D}{\mathcal{C}}\right],
	\eeq
	and
	\beq\label{bnd_infeas_s}
	\bbe[\|\Lb\hat \xb^N\|]\le ~ \tfrac{8\|\Lb\|}{N(N+3)} \left[3\sqrt{2\tilde D+ \mathbf{V}(\xb^0,\xb^*)}+\tfrac{7\|\Lb\|\mathcal{C}}{\mu}\|\yb^*-\yb^0\|\right],
	\eeq
where
$\hat\xb^N = \tfrac{2}{N(N+3)}\tsum_{k=1}^N(k+1)\hat\xb^k$.
\end{theorem}
\begin{proof}
	It is easy to check that \eqnok{stoch_s_para} satisfies conditions \eqnok{alpha_theta}-\eqnok{eta_tau_theta} and \eqref{theta_eta_ds}. Similarly, by \eqref{eqn:gapp}, Assumption~\eqref{assume:unbiased} and \eqref{assume:sm_bounded}, we can obtain
\begin{align}\label{g_bnd_exp_s}
\bbe[g(\hat \vb^N,\hat \zb^N)]
\le&~\left(\tsum_{k=1}^N\theta_k\right)^{-1}\Bigg\{\theta_1\eta_1\mathbf{V}(\xb^0,\xb^*)+\tfrac{\theta_1\tau_1}{2}\|\yb^0\|^2
+ \tsum_{k=1}^N\tsum_{t=1}^{T_k}\tsum_{i=1}^m\tfrac{2\theta_k}{T_k(T_k+1)}\left[
\tfrac{2t(M^2+\sigma^2)}{(t+1)\mu/\mathcal{C}+(t-1)\eta_k}\right]\Bigg\},
\end{align}
where $\vb^N=\left(\tsum_{k=1}^N\theta_k\right)^{-1}\hat \vb$.
Particularly, from \eqref{stoch_s_para}, we have
\begin{align*}
\tsum_{k=1}^N\tsum_{t=1}^{T_k}\tfrac{4m(M^2+\sigma^2)\theta_k}{T_k(T_k+1)}\tfrac{t} {(t+1)\mu/\mathcal{C}+(t-1)\eta_k}
&=\tsum_{k=1}^N\tfrac{4m(M^2+\sigma^2)\mathcal{C}\theta_k}{T_k(T_k+1)\mu} \tsum_{t=1}^{T_k}\tfrac{2t}{2(t+1)+(t-1)k}\\
&\qquad\le \tsum_{k=1}^N\tfrac{4m(M^2+\sigma^2)\mathcal{C}\theta_k}{T_k(T_k+1)\mu}\left(\tfrac{1}{2}+\tsum_{t=2}^{T_k}\tfrac{2t}{(t-1)(k+1)}\right)\\
&\qquad\le \tsum_{k=1}^N\tfrac{2m(M^2+\sigma^2)\mathcal{C}(k+1)}{T_k(T_k+1)\mu}+\tsum_{k=1}^N\tfrac{16m(M^2+\sigma^2)\mathcal{C}(T_k-1)}{T_k(T_k+1)\mu}
\le \tfrac{2\mu\tilde D}{\mathcal{C}}.
\end{align*}
Therefore, by plugging in these values into \eqref{g_bnd_exp_s}, we have
\begin{align}\label{g_bnd_exp_s1}
\bbe[g(\hat \vb^N,\hat \zb^N)]
\le&~\tfrac{2}{N(N+3)}\left[\tfrac{\mu}{\mathcal{C}} \mathbf{V}(\xb^0,\xb^*)+\tfrac{2\|L\|^2\mathcal{C}}{\mu}\|\yb^0\|^2+\tfrac{2\mu\tilde D}{\mathcal{C}}\right],
\end{align}
with
\[
\bbe[\|\hat \vb^N\|]=\tfrac{2}{N(N+3)}\bbe[\|\hat \vb\|]\le \tfrac{2\|\Lb\|}{N(N+3)}\bbe\left[(N+1)\|\hat \xb^N-\xb^{N-1}\|+\tfrac{4\|\Lb\|\mathcal{C}}{\mu}(\|\yb^N-\yb^*\|+\|\yb^*-\yb^0\|)\right].
\]
Note that from \eqref{eqn:xNbndstoch}, we have, for any $N\ge 2$,
\begin{align*}
\bbe[\|\hat\xb^N-\xb^{N-1}\|^2] \le \tfrac{8}{(N+1)(N-1)}\left[ \mathbf{V}(\xb^0,\xb^*) + \tfrac{2\|\Lb\|^2\mathcal{C}^2}{\mu^2}\|\yb^0 - \yb^*\|^2+2\tilde D\right],\\
\bbe[\|\yb^*-\yb^N\|^2]\le \tfrac{N \mu}{(N-1)\|\Lb\|^2\mathcal{C}}\left[\tfrac{\mu}{\mathcal{C}} \mathbf{V}(\xb^0,\xb^*) + \tfrac{2\|\Lb\|^2\mathcal{C}}{\mu}\|\yb^0 - \yb^*\|^2+\tfrac{2\mu\tilde D}{\mathcal{C}}\right].
\end{align*}	
Hence, in view of the above three relations and Jensen's inequality, we obtain
\begin{align*}
\bbe[\|\hat \vb^N\|]
\le &~ \tfrac{8\|\Lb\|}{N(N+3)} \left[3\sqrt{2\tilde D+ \mathbf{V}(\xb^0,\xb^*) + \tfrac{2\|\Lb\|^2\mathcal{C}^2}{\mu^2}\|\yb^0 - \yb^*\|^2}+\tfrac{\|\Lb\|\mathcal{C}}{\mu}\|\yb^*-\yb^0\|\right]\\\nn
\le &~ \tfrac{8\|\Lb\|}{N(N+3)} \left[3\sqrt{2\tilde D+ \mathbf{V}(\xb^0,\xb^*)}+\tfrac{7\|\Lb\|\mathcal{C}}{\mu}\|\yb^*-\yb^0\|\right].
\end{align*}
Applying Proposition~\ref{prop:approx} to the above inequality and \eqref{g_bnd_exp_s1}, the results in \eqref{bnd_func_s} and \eqref{bnd_infeas_s} follow immediately.
\end{proof}

	We now make some observations about the results obtained in Theorem~\ref{stoch_thm_s}. Firstly, similar to the general convex case, the best choice for $\tilde D$ (cf. \eqref{stoch_s_para}) would be $\mathbf{V}(\xb^0,\xb^*)$ so that the first and the third terms in \eqref{g_bnd_exp_s1} are about the same order. If there exists an estimate ${\cal D}_{X^m}>0$ satisfying \eqref{V_bndness},
	we can set $\tilde D={\cal D}_{X^m}^2$.

	Secondly, the complexity of SDCS method for solving strongly convex problems follows from \eqref{bnd_func_s} and \eqref{bnd_infeas_s}. Under the above assumption, with $\tilde D={\cal D}_{X^m}^2$ and $\yb^0=\0b$, the total number of inter-node communication rounds and intra-node subgradient evaluations performed by each agent for finding a stochastic $(\epsilon,\delta)$-solution of \eqref{eqn:prob} can be bounded by
	\beq\label{complexity_stoch_sc}
	{\cal O}\left\{\max\left(\sqrt{\tfrac{\mu{\cal D}_{X^m}^2}{\mathcal{C}\epsilon}},\sqrt{\tfrac{\|\Lb\|}{\delta}\left({\cal D}_{X^m}+\tfrac{\mathcal{C}\|\Lb\|\|\yb^*\|}{\mu}\right)}\right)\right\} \ \ \mbox{and} \ \
	{\cal O}\left\{\tfrac{m(M^2+\sigma^2)\mathcal{C}}{\mu}\max\left(\tfrac{1}{\epsilon}, \tfrac{\mathcal{C}\|\Lb\|}{\mu\delta}\left(\tfrac{1}{{\cal D}_{X^m}}+\tfrac{\mathcal{C}\|\Lb\|\|\yb^*\|}{{\cal D}_{X^m}^2\mu}\right)\right)\right\},
	\eeq
respectively.
In particular, if $\epsilon$ and $\delta$ satisfy \eqref{tolerance_delta_s}, the above complexity bounds, respectively, reduce to
\beq\label{simp_complexity_stoch_sc}
{\cal O}\left\{\sqrt{\tfrac{\mu{\cal D}_{X^m}^2}{\mathcal{C}\epsilon}}\right\}\ \
\mbox{and} \ \
{\cal O}\left\{\tfrac{m(M^2+\sigma^2)\mathcal{C}}{\mu\epsilon}\right\}.
\eeq
We can see that the total number of stochastic subgradient computations is
comparable to the optimal complexity bound obtained in \cite{GhaLan10-1,ghadimi2013optimal} for
stochastic strongly convex case in the centralized case.

\vgap
\subsection{High Probability Results}
All of the results stated in Section \ref{sec:sgc}-\ref{sec:ssc} are established in terms of expectation.
In order to provide high probability results for SDCS method, we additionally need the following ``light-tail'' assumption:
\begin{align}\label{assume:light_tail}
\bbe[\exp\{\|G_i(u^t,\xi_i^t)-f_i'(u^t)\|_*^2/\sigma^2\}]\le \exp\{1\}.
\end{align}
Note that \eqref{assume:light_tail} is stronger than \eqref{assume:sm_bounded}, since it implies \eqref{assume:sm_bounded} by Jensen's inequality.
Moreover, we also assume that there exists $\bar{\mathbf{V}}(\xb^*)$ s.t.
\beq\label{bnd_DX}
\bar{\mathbf{V}}(\xb^*):=\tsum_{i=1}^m\bar V_i(x_i^*):=\tsum_{i=1}^m\max_{x_i\in X_i}V_i(x_i^*,x_i).
\eeq

\vgap
The following theorem provides a large deviation result for the gap function $g(\hat \vb^N, \hat \zb^N)$ when our objective functions $f_i$, $i=1,\ldots, m$ are general nonsmooth convex functions.

\begin{theorem}\label{thm:main_stoch_prob}
	Assume the objective $f_i$, $i = 1, \ldots, m$ are general nonsmooth convex functions, i.e., $\mu = 0$ and $M >0$. Let Assumptions \eqref{assume:unbiased}, \eqref{assume:sm_bounded} and \eqref{assume:light_tail} hold, the parameters $\{\alpha_k\}$, $\{\theta_k\}$, $\{\eta_k\}$, $\{\tau_k\}$ and $\{T_k\}$ in Algorithm~\ref{alg:SDCS} satisfy \eqref{alpha_theta}-\eqref{eta_tau_theta}, and \eqref{theta_eta_d}, and the parameters $\{\lambda_t\}$ and $\{\beta_t\}$ in the CS procedure of Algorithm~\ref{alg:SDCS} be set as \eqref{eqn:csparam}.
	In addition, if $X_i$'s are compact, then for any $\zeta >0$ and $N \ge 1$, we have
	\begin{align}\label{gap_prob}
	\Prob\left\{g(\hat \vb^N,\hat \zb^N)\ge {\mathcal B}_d(N)+\zeta{\mathcal B}_p(N)\right\} \le \exp\{-\zeta^2/3\} + \exp\{-\zeta\},
	\end{align}
	where
	\begin{align}\label{def_Bd}
	\mathcal B_d(N):=&\left(\tsum_{k=1}^N\theta_k\right)^{-1}\Bigg[\tfrac{(T_1+1)(T_1+2)\theta_1\eta_1}{T_1(T_1+3)}\mathbf{V}(\xb^0,\xb^*)+\tfrac{\theta_1\tau_1}{2}\|\yb^0\|^2
	+ \tsum_{k=1}^N\tfrac{8m(M^2+\sigma^2)\theta_k}{\eta_k(T_k+3)}\Bigg],
	\end{align}
	and
		\begin{align}\label{def_Bp}
		\mathcal B_p(N):=&\left(\tsum_{k=1}^N\theta_k\right)^{-1}\left\{\sigma\left[2\bar{\mathbf{V}}(\xb^*)\tsum_{k=1}^N\tsum_{t=1}^{T_k}\left(\tfrac{\theta_k\lambda_t}{\tsum_{t=1}^{T_k}\lambda_t}\right)^2\right]^{1/2}
		+\tsum_{k=1}^N\tsum_{t=1}^{T_k}\tsum_{i=1}^m\tfrac{\sigma^2\theta_k\lambda_t}{\left(\tsum_{t=1}^{T_k}\lambda_t\right)\eta_k\beta_t}\right\}.
		\end{align}
\end{theorem}

In the next corollary, we establish the rate of convergence of SDCS in terms of both primal and feasibility (or consistency) residuals are of order ${\cal O}(1/N)$ with high probability when the objective functions are nonsmooth and convex.

	\begin{corollary}\label{cor:highprob}
	Let $\xb^*$ be an optimal solution of \eqnok{eqn:prob}, the parameters $\{\lambda_t\}$ and $\{\beta_t\}$ in the CS procedure of Algorithm~\ref{alg:SDCS} be set as \eqref{eqn:csparam}, and suppose that  $\{\alpha_k\}$, $\{\theta_k\}$,
	$\{\eta_k\}$, $\{\tau_k\}$ and $\{T_k\}$ are set to \eqref{stoch_ns_para} with $\tilde D=\bar{\mathbf{V}}(\xb^*)$. Under Assumptions~\eqref{assume:unbiased},~\eqref{assume:sm_bounded} and~\eqref{assume:light_tail}, we have for any $N \ge 1$ and $\zeta >0$
	\begin{align}\label{bnd_func_prob}
	\Prob&\left\{F(\hat \xb^N)-F(\xb^*)\ge \tfrac{\|\Lb\|}{N}\left[(7+8\zeta)\bar {\mathbf{V}}(\xb^*)+\tfrac{1}{2}\|\yb^0\|^2\right]\right\}
	\le \exp\{-\zeta^2/3\}+\exp\{-\zeta\},
	\end{align}
	and
	\begin{align}\label{bnd_infeas_prob}
	\Prob&\left\{\|\Lb\hat \xb^N\|^2\ge \tfrac{18\|\Lb\|^2}{N^2}\left[(7+8\zeta)\bar {\mathbf{V}}(\xb^*)+\tfrac{2}{3}\|\yb^*-\yb^0\|^2\right]\right\}
	\le \exp\{-\zeta^2/3\}+\exp\{-\zeta\}.
	\end{align}
\end{corollary}
\begin{proof}
	Observe that by the definition of $\lambda_t$ in \eqref{eqn:csparam},
	\begin{align*}
	\tsum_{t=1}^{T_k}\left[\tfrac{\theta_k\lambda_t}{\tsum_{t=1}^{T_k}\lambda_t}\right]^2&=\left(\tfrac{2}{T_k(T_k+3)}\right)^2\tsum_{t=1}^{T_k}(t+1)^2\\
	&=\left(\tfrac{2}{T_k(T_k+3)}\right)^2\tfrac{(T_k+1)(T_k+2)(2T_k+3)}{6}\le \tfrac{8}{3T_k},
	\end{align*}
	which together with \eqref{def_Bp} then imply that
		\begin{align*}
		{\cal B}_p(N)
		&\le\tfrac{1}{N}\left\{\sigma\left[2\bar {\mathbf{V}}(\xb^*)\tsum_{k=1}^N\tfrac{8}{3T_k}\right]^{1/2}+\tsum_{k=1}^N\tfrac{8m\sigma^2}{\|\Lb\|(T_k+3)}\right\}\\
		&\le \tfrac{4\|\Lb\|}{N }\left\{\sqrt{\tfrac{\bar {\mathbf{V}}(\xb^*)\tilde D}{3m}}+\tilde D\right\}
		\le \tfrac{8\|\Lb\|\bar {\mathbf{V}}(\xb^*)}{N }.
		\end{align*}
	Hence, \eqref{bnd_func_prob} follows from the above relation, \eqnok{gap_prob} and Proposition~\ref{prop:approx}.
Note that from \eqref{xNs_bnd} and plugging in \eqref{stoch_ns_para} with $\tilde D=\bar {\mathbf{V}}(\xb^*)$, we obtain
		\begin{align*}
		\|\hat \vb^N\|^2=&\left(\tsum_{k=1}^N\theta_k\right)^{-2}\|\hat \vb\|^2\\\nn
		\le&~\left(\tsum_{k=1}^N\theta_k\right)^{-2}\left\{3\theta_N^2\|\Lb\|^2\|\hat \xb^N-\xb^{N-1}\|^2+3\theta_1^2\tau_1^2\left(\|\yb^N-\yb^*\|^2+\|\yb^*-\yb^0\|^2\right)\right\}\\\nn
		\le &~\tfrac{3\|\Lb\|^2}{N^2}\Bigg\{18\mathbf{V}(\xb^0,\xb^*)+4\|\yb^*-\yb^0\|^2\\\nn
		&~~~+ \tsum_{k=1}^N\tsum_{t=1}^{T_k}\tsum_{i=1}^m\tfrac{12\theta_k}{T_k(T_k+3)\|\Lb\|}\left[(t+1)\langle\delta_i^{t-1,k}, x_i^*-u_i^{t-1}\rangle
		+\tfrac{4(M^2+\|\delta_i^{t-1,k}\|_*^2)}{\eta_k}\right]\Bigg\}.
		\end{align*}
	Hence, similarly, we have
		\begin{align*}
		\Prob&\left\{\|\hat \vb^N\|^2\ge \tfrac{18\|\Lb\|^2}{N^2}\left[(7+8\zeta)\bar {\mathbf{V}}(\xb^*)+\tfrac{2}{3}\|\yb^*-\yb^0\|^2\right]\right\}
		\le \exp\{-\zeta^2/3\}+\exp\{-\zeta\},
		\end{align*}
	which in view of Proposition~\ref{prop:approx} immediately implies \eqref{bnd_infeas_prob}.
\end{proof}

\setcounter{equation}{0}
\section{Convergence Analysis}
\label{sec:conv}
This section is devoted to prove the main lemmas in Section~\ref{sec:exact}, \ref{sec:deterministic} and~\ref{sec:stochastic}, which establish the convergence results of the decentralized primal-dual method, the deterministic and stochastic decentralized communication sliding methods, respectively. After introducing some general results about these algorithms, we provide the proofs for Lemma \ref{lem:main_exact}-\ref{lem:main_stoch_sc}
and Theorem \ref{thm:main_stoch_prob}.

The following lemma below characterizes the solution of the primal and dual projection steps
\eqref{eqn:algo2}, \eqref{eqn:algo3} (also \eqref{eqn:algo3-i}, \eqref{eqn:algo5-i}) as well as the projection in inner loop \eqref{eqn:inner1}. The proof of this result can be found in Lemma 2 of \cite{GhaLan10-1}.

\begin{lemma}\label{lem:optimality}
	Let the convex function $q: U \to \mathbb{R}$, the points $\bar{x},\bar{y}\in U$
	and the scalars $\mu_1,\mu_2 \in \mathbb{R}$ be given. Let $\omega: U\rightarrow \bbr$ be a differentiable convex function and $V(x,z)$ be defined in \eqref{eqn:def_V}. If
	\begin{align*}
	u^* \in \argmin\left\{ q(u) + \mu_1V(\bar{x}, u) + \mu_2V(\bar{y},u): u\in U\right\},
	\end{align*}
	then for any $u \in U$, we have
	\begin{align*}
	q(u^*) +\mu_1V(\bar{x},u^*) + \mu_2V(\bar{y},u^*)
	&\le q(u) +\mu_1V(\bar{x},u)+ \mu_2V(\bar{y},u) - (\mu_1 + \mu_2)V(u^*,u).
	\end{align*}
\end{lemma}

We are now ready to provide a proof for Lemma~\ref{lem:main_exact} which establishes the convergence property for the decentralized primal-dual method. Note that this result also builds up the basic recursion for the outer loop of the DCS and SDCS methods.
\vgap

\noindent{\bf Proof of Lemma~\ref{lem:main_exact}:}
	Note that applying Lemma~\ref{lem:optimality} to \eqref{eqn:algo3-i} and \eqref{eqn:algo5-i}, we have
	\[
	\la v_i^k, y_i-y_i^k\ra \le \tfrac{\tau_k}{2}\left[\|y_i-y_i^{k-1}\|^2-\|y_i-y_i^k\|^2-\|y_i^{k-1}-y_i^k\|^2\right], \ \forall y_i\in \bbr^d,
	\]
	\[
	\la w_i^k, x_i^k-x_i\ra + f_i(x_i^k)-f_i(x_i)\le
	\eta_k\left[V_i(x_i^{k-1}, x_i)-V_i(x_i^{k}, x_i)-V_i(x_i^{k-1}, x_i^k)\right], \ \forall x_i\in X_i,
	\]
	which in view of the definition of $Q$ and $\mathbf{V}(\cdot,\cdot)$ in \eqref{eqn:gap} and \eqref{eqn:def_Vb}, respectively, we can obtain
	\begin{align*}
	Q(\xb^k,\yb^k; \zb)
	&= F(\xb^k) - F(\xb)+\la \Lb\xb^k,\yb \ra -\la \Lb \xb, \yb^k\ra\nn\\
	&\le \la \Lb (\xb^k-\tilde \xb^{k}),\yb-\yb^k\ra
	+\eta_k\left[\mathbf{V}(\xb^{k-1}, \xb)-\mathbf{V}(\xb^{k}, \xb)-\mathbf{V}(\xb^{k-1}, \xb^k)\right]\nn\\
	&~~~+\tfrac{\tau_k}{2}\left[\|\yb-\yb^{k-1}\|^2-\|\yb-\yb^k\|^2-\|\yb^{k-1}-\yb^k\|^2\right], \ \forall \zb\in X^m\times \bbr^{md}.
	\end{align*}
    Multiplying both sides of the above inequality by $\theta_k$, and summing the resulted inequality from $k=1$ to $N$, we obtain
    \beq\label{Q_rec}
    \tsum_{k=1}^N\theta_kQ(\xb^k,\yb^k;\zb)\le \tsum_{k=1}^N\theta_k\Delta_k,
    \eeq
    where
    \begin{align}\label{def_Delta}
    \Delta_k:=~&\la \Lb (\xb^k-\tilde \xb^{k}),\yb-\yb^k\ra
    +\eta_k\left[\mathbf{V}(\xb^{k-1}, \xb)-\mathbf{V}(\xb^{k}, \xb)-\mathbf{V}(\xb^{k-1}, \xb^k)\right]\\\nn
    &~~~ + \tfrac{\tau_k}{2}\left[\|\yb-\yb^{k-1}\|^2-\|\yb-\yb^k\|^2-\|\yb^{k-1}-\yb^k\|^2\right].
    \end{align}
    Observe that from the definition of $\tilde \xb ^k$ in \eqref{eqn:algo1}, \eqref{theta_eta} and \eqref{theta_tau}, we have
    \begin{align}\label{Delta_recursion}
    \tsum_{k=1}^N\theta_k\Delta_k
    &=\tsum_{k=1}^N\left[\theta_k\la \Lb(\xb^k-\xb^{k-1}), \yb-\yb^k\ra -\alpha_k\theta_k\la \Lb(\xb^{k-1}-\xb^{k-2}), \yb-\yb^{k-1}\ra\right]\\\nn
    &~~~ -\tsum_{k=1}^N\theta_k\left[\alpha_k\la \Lb(\xb^{k-1}-\xb^{k-2}), \yb^{k-1}-\yb^{k}\ra + \eta_k\mathbf{V}(\xb^{k-1},\xb^k)
    +\tfrac{\tau_k}{2}\|\yb^{k-1}-\yb^k\|^2\right]\\\nn
    &~~~+\tsum_{k=2}^N(\theta_k\eta_k-\theta_{k-1}\eta_{k-1})\mathbf{V}(\xb^{k-1},\xb)
    +\theta_1\eta_1\mathbf{V}(\xb^0,\xb)-\theta_N\eta_N\mathbf{V}(\xb^N,\xb)
    \\\nn
    &~~~+\tsum_{k=2}^N(\tfrac{\theta_k\tau_k}{2}-\tfrac{\theta_{k-1}\tau_{k-1}}{2})\|\yb-\yb^{k-1}\|^2+\tfrac{\theta_1\tau_1}{2}\|\yb-\yb^0\|^2-\tfrac{\theta_N\tau_N}{2}\|\yb-\yb^N\|^2\\\nn
    &\le \tsum_{k=1}^N\left[\theta_k\la \Lb(\xb^k-\xb^{k-1}), \yb-\yb^k\ra -\alpha_k\theta_k\la \Lb(\xb^{k-1}-\xb^{k-2}), \yb-\yb^{k-1}\ra\right]\\\nn
    &~~~ -\tsum_{k=1}^N\theta_k\left[\alpha_k\la \Lb(\xb^{k-1}-\xb^{k-2}), \yb^{k-1}-\yb^{k}\ra + \eta_k\mathbf{V}(\xb^{k-1},\xb^k)
    +\tfrac{\tau_k}{2}\|\yb^{k-1}-\yb^k\|^2\right]\\\nn
    &~~~+\theta_1\eta_1\mathbf{V}(\xb^0,\xb)-\theta_N\eta_N\mathbf{V}(\xb^N,\xb)
    +\tfrac{\theta_1\tau_1}{2}\|\yb-\yb^0\|^2-\tfrac{\theta_N\tau_N}{2}\|\yb-\yb^N\|^2\\\nn
    &\myeqa \theta_N \la \Lb(\xb^N-\xb^{N-1}), \yb-\yb^N\ra -\theta_N\eta_N\mathbf{V}(\xb^{N-1},\xb^N)\\\nn
    &~~~-\tsum_{k=2}^N\left[\theta_k\alpha_k\la \Lb (\xb^{k-1}-\xb^{k-2}),\yb^{k-1}-\yb^k\ra +\theta_{k-1}\eta_{k-1}\mathbf{V}(\xb^{k-2},\xb^{k-1})
    +\tfrac{\theta_k\tau_k}{2}\|\yb^{k-1}-\yb^k\|^2\right]\\\nn
    &~~~+\theta_1\eta_1\mathbf{V}(\xb^0,\xb)-\theta_N\eta_N\mathbf{V}(\xb^N,\xb)
    +\tfrac{\theta_1\tau_1}{2}\|\yb-\yb^0\|^2-\tfrac{\theta_N\tau_N}{2}\|\yb-\yb^N\|^2\\\nn
    &\myeqb \theta_N \la \Lb(\xb^N-\xb^{N-1}), \yb-\yb^N\ra -\theta_N\eta_N\mathbf{V}(\xb^{N-1},\xb^N)\\\nn
    &~~~+\tsum_{k=2}^N\left(\tfrac{\theta_{k-1}\alpha_k\|\Lb\|^2}{2\tau_k}-\tfrac{\theta_{k-1}\eta_{k-1}}{2}\right)\|\xb^{k-2}-\xb^{k-1}\|^2\\\nn
    &~~~+\theta_1\eta_1\mathbf{V}(\xb^0,\xb)-\theta_N\eta_N\mathbf{V}(\xb^N,\xb)
    +\tfrac{\theta_1\tau_1}{2}\|\yb-\yb^0\|^2-\tfrac{\theta_N\tau_N}{2}\|\yb-\yb^N\|^2\\\nn
    &\myeqc \theta_N \la \Lb(\xb^N-\xb^{N-1}), \yb-\yb^N\ra -\theta_N\eta_N\mathbf{V}(\xb^{N-1},\xb^N)\\\nn
    &~~~+\theta_1\eta_1\mathbf{V}(\xb^0,\xb)-\theta_N\eta_N\mathbf{V}(\xb^N,\xb)
    +\tfrac{\theta_1\tau_1}{2}\|\yb-\yb^0\|^2-\tfrac{\theta_N\tau_N}{2}\|\yb-\yb^N\|^2\\\nn
    &\myeqd \theta_N \la \yb^N, \Lb(\xb^{N-1}-\xb^N)\ra -\theta_N\eta_N\mathbf{V}(\xb^{N-1},\xb^N)-\tfrac{\theta_1\tau_1}{2}\|\yb^N\|^2\\\nn
    &~~~+\theta_1\eta_1\mathbf{V}(\xb^0,\xb)+\tfrac{\theta_1\tau_1}{2}\|\yb^0\|^2
    +\la \yb, \theta_N\Lb(\xb^N-\xb^{N-1}) + \theta_1\tau_1(\yb^N-\yb^0)\ra,\\\nn
    &\myeqe \left(\tfrac{\theta_N\|\Lb\|^2}{2\eta_N }-\tfrac{\theta_1\tau_1}{2}\right)\|\yb^N\|^2
    +\theta_1\eta_1\mathbf{V}(\xb^0,\xb)+\tfrac{\theta_1\tau_1}{2}\|\yb^0\|^2
    +\la \yb, \theta_N\Lb(\xb^N-\xb^{N-1}) + \theta_1\tau_1(\yb^N-\yb^0)\ra,
    \end{align}
    where (a) follows from \eqref{alpha_theta} and the fact that $\xb^{-1} = \xb^0$,
    (b) follows from the simple relation that $b\la u,v\ra -a\|v\|^2/2\le b^2\|u\|^2/(2a), \forall a>0$, \eqref{alpha_theta} and \eqref{eqn:def_V_s},
    (c) follows from \eqref{eta_tau_L_k},
    (d) follows from \eqref{eta_tau}, $\|\yb-\yb^0\|^2 - \|\yb-\yb^N\|^2 = \|\yb^0\|^2 - \|\yb^N\|^2 -2 \la \yb, \yb^0-\yb^N\ra$ and arranging the terms accordingly,
    (e) follows from \eqref{eqn:def_V_s} and the relation $b\la u,v\ra -a\|v\|^2/2\le b^2\|u\|^2/(2a), \forall a>0$.
    The desired result in \eqref{outer_recursion} then follows from this relation, \eqref{eta_tau_theta}, \eqref{Q_rec} and the convexity of $Q$.

	Furthermore, from \eqref{Delta_recursion}(c), \eqref{eqn:def_V_s} and the fact that $\tsum_{k=1}^N\theta_kQ(\xb^k,\yb^k;\zb^*)\ge 0$, if we fix $\zb=\zb^*=(\xb^*,\yb^*)$ in the above relation, we have
	\begin{align*}
	\tfrac{\theta_N\tau_N }{2}\|\xb^{N-1}-\xb^{N}\|^2
	&\le \theta_N \la \Lb(\xb^N-\xb^{N-1}), \yb^*-\yb^N\ra -\tfrac{\theta_N\tau_N}{2}\|\yb^*-\yb^N\|^2
	+\theta_1\eta_1\mathbf{V}(\xb^0,\xb^*)+\tfrac{\theta_1\tau_1}{2}\|\yb^*-\yb^0\|^2\\
	&\le \tfrac{\theta_N\|\Lb\|^2}{2\tau_N}\|\xb^{N-1}-\xb^N\|^2+\theta_1\eta_1\mathbf{V}(\xb^0,\xb^*)+\tfrac{\theta_1\tau_1}{2}\|\yb^*-\yb^0\|^2,
	\end{align*}
	Similarly, we obtain
	\begin{align*}
	\tfrac{\theta_N\tau_N}{2}\|\yb^*-\yb^N\|^2
	&\le \tfrac{\theta_N\|\Lb\|^2}{2\eta_N }\|\yb^*-\yb^N\|^2+\theta_1\eta_1\mathbf{V}(\xb^0,\xb^*)+\tfrac{\theta_1\tau_1}{2}\|\yb^*-\yb^0\|^2,
	\end{align*}
	from which the desired result in \eqref{xNs_bnd_pd} follows.
    \hfill $\square$

Before we provide proofs for the remaining lemmas, 
we first need to present a result which summarizes an important convergence property of the \SGD~procedure.
It needs to be mentioned that the following proposition states a general result holds for \SGD~procedure performed by individual agent $i \in \Vc$. For notation convenience, we use the notations defined in \SGD~procedure (cf. Algorithm~\ref{alg:DCS}).

\begin{proposition}\label{prop:inner}
	If $\{\b_t\}$ and $\{\lambda_t\}$ in the \SGD~procedure satisfy
	\beq\label{beta_w}
	\lambda_{t+1}(\eta\beta_{t+1}-\mu/\mathcal{C})\le \lambda_{t}(1+\beta_{t})\eta, \ \forall t\ge 1.
	\eeq
	then, for $t \ge 1$ and $u \in U$,
	\begin{align}\label{inner_recursion}
	&(\tsum_{t=1}^T\lambda_t)^{-1}\left[\eta(1+\beta_T)\lambda_TV(u^T,u) + \tsum_{t=1}^T\lambda_t\langle\delta^{t-1}, u-u^{t-1}\rangle\right] + \Phi(\hat{u}^T) - \Phi(u)\\\nn
	&~\le  (\tsum_{t=1}^T\lambda_t)^{-1}\left[(\eta\beta_1-\mu/\mathcal{C})\lambda_1V(u^0,u)
	+ \tsum_{t=1}^T\tfrac{(M+\|\delta^{t-1}\|_*)^2\lambda_t}{2 \eta\beta_t}\right],
	\end{align}
	where $\Phi$ is defined as
\begin{align}\label{eqn:defphi}
\Phi(u) := \la w, u\ra + \phi(u) + \eta V(x,u)
\end{align}
 and $\delta^t:=\phi'(u^t)-h^t$.
\end{proposition}
\begin{proof}
	Noticing that $\phi := f_i$ in the CS procedure, we have by \eqref{eqn:nonsmooth}
	\begin{align*}
	\phi(u^t)
	\le &~ \phi(u^{t-1}) + \la  \phi'(u^{t-1}), u^t - u^{t-1}\ra + M \|u^t-u^{t-1}\|\\
	= &~ \phi(u^{t-1}) + \la  \phi'(u^{t-1}), u - u^{t-1}\ra + \la \phi'(u^{t-1}), u^t - u\ra + M \|u^t-u^{t-1}\|\\
	\le &~ \phi(u) - \tfrac{\mu}{2}\|u-u^{t-1}\|^2 + \la  \phi'(u^{t-1}), u^t - u\ra + M \|u^t-u^{t-1}\|,
	\end{align*}
where $\phi'(u^{t-1})\in \partial \phi(u^{t-1})$ and $\partial \phi(u^{t-1})$ denotes the subdifferential of $\phi$ at $u^{t-1}$.
	By applying Lemma \ref{lem:optimality} to \eqref{eqn:inner1}, we obtain
	\begin{align*}
	&\la w + h^{t-1}, u^t-u \ra +\eta V(x, u^t) -\eta V(x,u)
	\le  \eta\b_tV(u^{t-1},u)-\eta(1+\b_t)V(u^t,u)-\eta\b_tV(u^{t-1},u^{t}),\ \forall u\in U.
	\end{align*}
	Combining the above two relations together with \eqref{eqn:proxquad}
	\footnote{Observed that we only need condition \eqref{eqn:proxquad} when $\mu>0$, in other words, the objective functions $f_i$'s are strongly convex.}, we conclude that
	\begin{align}\label{eqn:opt_1}
	&\la w ,u^t- u \ra + \phi(u^t)- \phi(u) + \la \delta^{t-1}, u-u^{t-1}\ra +\eta V(x,u^t)- \eta V(x,u)\\\nn
	&~\le  (\eta\b_t-\mu/\mathcal{C}) V(u^{t-1},u)-\eta(1+\b_t)V(u^t,u)
	+ \la \delta^{t-1},u^t-u^{t-1}\ra
	+ M\|u^t-u^{t-1}\|- \eta\beta_tV(u^{t-1},u^{t}), \ \forall u\in U.
	\end{align}
	Moreover, by Cauchy-Schwarz inequality, \eqref{eqn:def_V_i_s}, and the simple fact that $-at^2/2 + bt \le b^2/(2a)$ for any $a>0$, we have
	\begin{align*}
	&\la \delta^{t-1},u^t-u^{t-1}\ra + M\|u^t-u^{t-1}\|- \eta\beta_tV(u^{t-1},u^{t})
	\le ( \|\delta^{t-1}\|_* + M)\|u^t-u^{t-1}\|- \tfrac{\eta\beta_t }{2}\|u^t-t^{t-1}\|^2\le \tfrac{(M+\|\delta^{t-1}\|_*)^2}{2 \eta\b_t}.
	\end{align*}
	From the above relation  and the definition of $\Phi(u)$ in \eqref{eqn:defphi}, we can rewrite \eqref{eqn:opt_1} as,
	\begin{align*}
	\Phi(u^t) - \Phi(u)&+\la \delta^{t-1}, u-u^{t-1}\ra
	\le   (\eta\b_t-\mu/\mathcal{C})V(u^{t-1},u)-\eta(1+\b_t)V(u^t,u)
	+ \tfrac{(M+\|\delta^{t-1}\|_*)^2}{2 \eta\b_t}, \ \forall u\in U.
	\end{align*}
	Multiplying both sides by $\lambda_t$ and summing up the resulting inequalities from $t=1$ to $T$, we obtain
	\begin{align*}
	&\tsum_{t=1}^T\lambda_t\left[\Phi(u^t) - \Phi(u)+\la \delta^{t-1}, u-u^{t-1}\ra\right]
	\le   \tsum_{t=1}^T\left[(\eta\b_t-\mu/\mathcal{C})\lambda_tV(u^{t-1},u)-\eta(1+\beta_t)\lambda_tV(u^t,u)\right]
	+ \tsum_{t=1}^T\tfrac{(M+\|\delta^{t-1}\|_*)^2\lambda_t}{2 \eta\b_t}.
	\end{align*}
	Hence, in view of \eqref{beta_w}, the convexity of $\Phi$ and the definition of $\hat u^T$ in \eqref{eqn:inner2}, we have
	\begin{align*}
	\Phi(\hat u^T) - \Phi(u)+ &(\tsum_{t=1}^T\lambda_t)^{-1}\tsum_{t=1}^T\lambda_t\la \delta^{t-1}, u-u^{t-1}\ra\\
	&~\le (\tsum_{t=1}^T\lambda_t)^{-1} \left[(\eta\beta_1-\mu/\mathcal{C})\lambda_1V(u^0,u)-\eta(1+\beta_T)\lambda_TV(u^T,u)
	+ \tsum_{t=1}^T\tfrac{(M+\|\delta^{t-1}\|_*)^2\lambda_t}{2 \eta\b_t}\right],
	\end{align*}
	which implies \eqref{inner_recursion} immediately.
\end{proof}

As a matter of fact, the SDCS method covers the DCS method as a special case when $\delta^t = 0,\ \forall t\ge0$. Therefore, we investigate the proofs for Lemma~\ref{lem:main_stoch} and \ref{lem:main_stoch_sc}  first and then simplify the proofs for Lemma~\ref{lem:main_deter} and \ref{lem:main_deter_sc}.

\vgap
We now provide a proof for Lemma~\ref{lem:main_stoch}, which establishes the convergence property of SDCS method for solving general convex problems.
\vgap

\noindent{\bf Proof of Lemma~\ref{lem:main_stoch}}\\
When $f_i$, $i=1,\dots,m$, are general convex functions, we have $\mu=0$ and $M>0$ (cf. \eqref{eqn:nonsmooth}). Therefore, in view of $\phi := f_i$, and $\lambda_t$ and $\beta_t$ defined in \eqref{eqn:csparam} satisfying condition \eqref{beta_w} in the CS procedure, equation \eqref{inner_recursion} can be rewritten as the following,\footnote{We added the subscript $i$ to emphasize that this inequality holds for any agent $i \in \Vc$ with $\phi = f_i$. More specifically, $\Phi_i(u_i) := \la w_i, u_i\ra + f_i(u_i) + \eta V_i(x_i,u_i)$.}
	\begin{align*}
	&(\tsum_{t=1}^T\lambda_t)^{-1}\left[\eta(1+\beta_T)\lambda_TV_i(u_i^T,u_i) + \tsum_{t=1}^T\lambda_t\langle\delta_i^{t-1}, u_i-u_i^{t-1}\rangle\right] + \Phi_i(\hat{u}_i^T) - \Phi_i(u_i)\\\nn
	&~\le  (\tsum_{t=1}^T\lambda_t)^{-1}\left[\eta\beta_1\lambda_1V_i(u_i^0,u_i)
	+ \tsum_{t=1}^T\tfrac{(M+\|\delta_i^{t-1}\|_*)^2\lambda_t}{2 \eta\beta_t}\right], \ \forall u_i\in X_i.
	\end{align*}
In view of the above relation, the definition of $\Phi^k$ in \eqref{eqn:defPhi_k}, and the input and output settings in the \SGD~procedure, it is not difficult to see that, for any $k \ge 1$,\footnote{We added the superscript $k$ in $\delta_i^{t-1,k}$ to emphasize that this error is generated at the $k$-th outer loop.}
	\begin{align*}
	\Phi^k(\hat{\xb}^k) - \Phi^k(\xb) &~+(\tsum_{t=1}^{T_k}\lambda_{t})^{-1}\left[\eta_k(1+\beta_{T_k})\lambda_{T_k}\mathbf{V}(\xb^k,\xb)+\tsum_{t=1}^{T_k}\tsum_{i=1}^m\lambda_t\langle\delta_i^{t-1,k}, x_i-u_i^{t-1}\rangle\right] \\
	&~~~\le  (\tsum_{t=1}^{T_k}\lambda_t)^{-1}\left[\eta_k\beta_1\lambda_1\mathbf{V}(\xb^{k-1},\xb)
	+ \tsum_{t=1}^{T_k}\tsum_{i=1}^m\tfrac{(M+\|\delta_i^{t-1,k}\|_*)^2\lambda_t}{2 \eta_k\beta_t}\right],\ \forall \xb\in X^m.
	\end{align*}
	By plugging into the above relation the values of $\lambda_t$ and $\beta_t$ in \eqref{eqn:csparam}, together with the definition of $\Phi^k$ in \eqref{eqn:defPhi_k} and rearranging the terms, we have,
	\begin{align*}
	\la \Lb(\hat \xb^k-\xb), \yb^k\ra + F(\hat \xb^k)- F(\xb)
	&\le \tfrac{(T_k+1)(T_k+2)\eta_k}{T_k(T_k+3)}\left[\mathbf{V}(\xb^{k-1},\xb) - \mathbf{V}(\xb^k,\xb)\right]
	-\eta_k\mathbf{V}(\xb^{k-1},\hat \xb^k)\\
	&~+\tfrac{2}{T_k(T_k+3)}\tsum_{t=1}^{T_k}\tsum_{i=1}^m\left[(t+1)\langle\delta_i^{t-1,k}, x_i-u_i^{t-1}\rangle
	+\tfrac{2(M+\|\delta_i^{t-1,k}\|_*)^2}{\eta_k}\right],
	\ \forall \xb\in X^m.
	\end{align*}
	Moreover, applying Lemma~\ref{lem:optimality} to \eqref{eqn:algo3-cs}, we have, for $k\ge 1$,
	\beq\label{dual_opt}
	\la v_i^k, y_i-y_i^k\ra \le \tfrac{\tau_k}{2}\left[\|y_i-y_i^{k-1}\|^2-\|y_i-y_i^k\|^2-\|y_i^{k-1}-y_i^k\|^2\right], \ \forall y_i\in \bbr^d,
	\eeq
	which in view of the definition of $Q$ in \eqref{eqn:gap} and the above two relations, then implies that, for $k\ge 1$, $\zb\in X^m\times \bbr^{md}$,
	\begin{align*}
	Q(\hat \xb^k,\yb^k; \zb)
	&= F(\hat \xb^k) - F(\xb)+\la \Lb\hat \xb^k,\yb \ra -\la \Lb \xb, \yb^k\ra\\
	&\le \la \Lb (\hat \xb^k-\tilde \xb^{k}),\yb-\yb^k\ra
	+\tfrac{(T_k+1)(T_k+2)\eta_k}{T_k(T_k+3)}\left[\mathbf{V}(\xb^{k-1},\xb)-\mathbf{V}(\xb^k,\xb)\right]\\
	&~~~-\eta_k\mathbf{V}(\xb^{k-1},\hat \xb^k) +\tfrac{\tau_k}{2}\left[\|\yb-\yb^{k-1}\|^2-\|\yb-\yb^k\|^2
	-\|\yb^{k-1}-\yb^k\|^2\right]\\
	&~~~+\tfrac{2}{T_k(T_k+3)}\tsum_{t=1}^{T_k}\tsum_{i=1}^m\left[(t+1)\langle\delta_i^{t-1,k}, x_i-u_i^{t-1}\rangle
	+\tfrac{2(M+\|\delta_i^{t-1,k}\|_*)^2}{\eta_k}\right].
	\end{align*}
	Multiplying both sides of the above inequality by $\theta_k$, and summing up the resulting inequalities from $k=1$ to $N$, we obtain, for all $\zb \in X^m\times \bbr^{md}$,
	\begin{align}\label{Q_tilde_D}
	\tsum_{k=1}^N\theta_kQ(\hat \xb^k,\yb^k;\zb)
	&\le \tsum_{k=1}^N\theta_k\tilde \Delta_k + \tsum_{k=1}^N\tsum_{t=1}^{T_k}\tsum_{i=1}^m\tfrac{2\theta_k}{T_k(T_k+3)}
	\left[(t+1)\langle\delta_i^{t-1,k}, x_i-u_i^{t-1}\rangle+\tfrac{2(M+\|\delta_i^{t-1,k}\|_*)^2}{\eta_k}\right],
	\end{align}
	where
	\begin{align}\label{def_tDelta}
	\tilde \Delta_k:=&\la \Lb (\hat \xb^k-\tilde \xb^{k}),\yb-\yb^k\ra
	+\tfrac{(T_k+1)(T_k+2)\eta_k}{T_k(T_k+3)}\left[\mathbf{V}(\xb^{k-1},\xb)-\mathbf{V}(\xb^k,\xb)\right]\\\nn
	&~~~-\eta_k\mathbf{V}(\xb^{k-1},\hat \xb^k)+\tfrac{\tau_k}{2}\left[\|\yb-\yb^{k-1}\|^2-\|\yb-\yb^k\|^2-\|\yb^{k-1}-\yb^k\|^2\right].
	\end{align}
	Since $\tilde \Delta_k$ in \eqref{def_tDelta} shares a similar structure with $\Delta_k$ in \eqref{def_Delta} (with $\xb^k$ in the first and the fourth terms being replaced by $\hat\xb^k$),
	we can follow the procedure in \eqref{Delta_recursion} to simplify the RHS of \eqref{Q_tilde_D}. The only difference is in the coefficient of the term $[\mathbf{V}(\xb^{k-1},\xb) - \mathbf{V}(\xb^k,\xb)]$. Hence, by using condition \eqref{theta_eta_d} in place of \eqref{theta_eta}, 
we obtain
	\begin{align}\label{Q_bnd}
	\tsum_{k=1}^N\theta_kQ(\hat \xb^k,\yb^k;\zb)
	&\le \tsum_{k=1}^N\theta_k\tilde \Delta_k + \tsum_{k=1}^N\tsum_{t=1}^{T_k}\tsum_{i=1}^m\tfrac{2\theta_k}{T_k(T_k+3)}
	\left[(t+1)\langle\delta_i^{t-1,k}, x_i-u_i^{t-1}\rangle+\tfrac{2(M+\|\delta_i^{t-1,k}\|_*)^2}{\eta_k}\right]\\\nn
	&\le \tfrac{(T_1+1)(T_1+2)\theta_1\eta_1}{T_1(T_1+3)}\mathbf{V}(\xb^0,\xb)+\tfrac{\theta_1\tau_1}{2}\|\yb^0\|^2+\la\hat \vb,\yb\ra\\\nn
	&~~~+ \tsum_{k=1}^N\tsum_{t=1}^{T_k}\tsum_{i=1}^m\tfrac{2\theta_k}{T_k(T_k+3)}\left[(t+1)\langle\delta_i^{t-1,k}, x_i-u_i^{t-1}\rangle
	+\tfrac{4(M^2+\|\delta_i^{t-1,k}\|_*^2)}{\eta_k}\right], \ \forall \zb \in X^m\times \bbr^{md},
	\end{align}
where
\begin{align}\label{eqn:hatvb}
\hat\vb := \theta_N \Lb (\hat \xb^N - \xb^{N-1}) + \theta_1\tau_1 (\yb^N - \yb^0).
\end{align}
Our result in \eqref{Q_pert_bnd} immediately follows from the convexity of $Q$.

	Furthermore, in view of \eqref{Delta_recursion}(c) and \eqref{Q_tilde_D}, we can obtain the following similar result (with $\xb^N$ in the first and the second terms of the RHS being replaced with $\hat\xb^N$),
	\begin{align*}
	\tsum_{k=1}^N\theta_kQ(\hat \xb^k,\yb^k;\zb)
	&\le \theta_N \la \Lb(\hat \xb^N-\xb^{N-1}), \yb-\yb^N\ra -\theta_N\eta_N\mathbf{V}(\xb^{N-1},\hat \xb^N)\\\nn
	&~~~+\tfrac{(T_1+1)(T_1+2)\theta_1\eta_1}{T_1(T_1+3)}\mathbf{V}(\xb^0,\xb)-\tfrac{(T_N+1)(T_N+2)\theta_N\eta_N}{T_N(T_N+3)}\mathbf{V}(\xb^N,\xb)
	+\tfrac{\theta_1\tau_1}{2}\|\yb-\yb^0\|^2-\tfrac{\theta_N\tau_N}{2}\|\yb-\yb^N\|^2\\\nn
	&~~~+ \tsum_{k=1}^N\tsum_{t=1}^{T_k}\tsum_{i=1}^m\tfrac{\theta_k}{T_k(T_k+3)}\left[(t+1)\langle\delta_i^{t-1,k}, x_i-u_i^{t-1}\rangle
	+\tfrac{4(M^2+\|\delta_i^{t-1,k}\|_*^2)}{\eta_k}\right].
	\end{align*}
	 Therefore, in view of the fact that $\tsum_{k=1}^N\theta_kQ(\hat\xb^k,\yb^k;\zb^*)\ge 0$ for any saddle point $\zb^*=(\xb^*,\yb^*)$ of \eqref{eqn:saddle}, and \eqref{eqn:def_V_s}, by fixing $\zb=\zb^*$ and rearranging terms, we obtain
	\begin{align}\label{bnd_xNs_ns}
	\tfrac{\theta_N\eta_N }{2}\|\hat \xb^N-\xb^{N-1}\|^2
	&\le \theta_N \la \Lb(\hat \xb^N-\xb^{N-1}), \yb^*-\yb^N\ra -\tfrac{\theta_N\tau_N}{2}\|\yb^*-\yb^N\|^2\\\nn
	&~~~+\tfrac{(T_1+1)(T_1+2)\theta_1\eta_1}{T_1(T_1+3)}\mathbf{V}(\xb^0, \xb^*)+\tfrac{\theta_1\tau_1}{2}\|\yb^*-\yb^0\|^2\\\nn
	&~~~+ \tsum_{k=1}^N\tsum_{t=1}^{T_k}\tsum_{i=1}^m\tfrac{2\theta_k}{T_k(T_k+3)}\left[(t+1)\langle\delta_i^{t-1,k}, x_i^*-u_i^{t-1}\rangle
	+\tfrac{4(M^2+\|\delta_i^{t-1,k}\|_*^2)}{\eta_k}\right]\\\nn
	&\le \tfrac{\theta_N\|\Lb\|^2}{2\tau_N}\|\hat \xb^N-\xb^{N-1}\|^2 +\tfrac{(T_1+1)(T_1+2)\theta_1\eta_1}{T_1(T_1+3)}\mathbf{V}(\xb^0, \xb^*)
    +\tfrac{\theta_1\tau_1}{2}\|\yb^*-\yb^0\|^2\\\nn
	&~~~+ \tsum_{k=1}^N\tsum_{t=1}^{T_k}\tsum_{i=1}^m\tfrac{2\theta_k}{T_k(T_k+3)}\left[(t+1)\langle\delta_i^{t-1,k}, x_i^*-u_i^{t-1}\rangle
	+\tfrac{4(M^2+\|\delta_i^{t-1,k}\|_*^2)}{\eta_k}\right],
	\end{align}
	where the second inequality follows from the relation $b\la u,v\ra -a\|v\|^2/2\le b^2\|u\|^2/(2a), \forall a>0$.
	
	Similarly, we obtain
	\begin{align}\label{bnd_yNs_ns}
	\tfrac{\theta_N\tau_N}{2}\|\yb^*-\yb^N\|^2
	&\le \tfrac{\theta_N\|\Lb\|^2}{2\eta_N }\|\yb^*-\yb^N\|^2+\tfrac{(T_1+1)(T_1+2)\theta_1\eta_1}{T_1(T_1+3)}\mathbf{V}(\xb^0, \xb^*)
    +\tfrac{\theta_1\tau_1}{2}\|\yb^*-\yb^0\|^2\\\nn
	&~~~+ \tsum_{k=1}^N\tsum_{t=1}^{T_k}\tsum_{i=1}^m\tfrac{2\theta_k}{T_k(T_k+3)}\left[(t+1)\langle\delta_i^{t-1,k}, x_i^*-u_i^{t-1}\rangle
	+\tfrac{4(M^2+\|\delta_i^{t-1,k}\|_*^2)}{\eta_k}\right],
	\end{align}
	from which the desired result in \eqref{xNs_bnd} follows.
\\

The following proof of Lemma~\ref{lem:main_stoch_sc} establishes the convergence of SDCS method for solving strongly convex problems.

\vgap
	
\noindent{\bf Proof of Lemma~\ref{lem:main_stoch_sc}}\\
When $f_i$, $i=1,\dots,m$, are strongly convex functions, we have $\mu,\ M>0$ (cf. \eqref{eqn:nonsmooth}). Therefore, in view of Proposition~\ref{prop:inner} with $\lambda_t$ and $\beta_t$ defined in \eqref{eqn:csparam_sc} satisfying condition \eqref{beta_w}, the definition of $\Phi^k$ in \eqref{eqn:defPhi_k}, and the input and output settings in the \SGD~procedure, we have for all $k \ge 1$
	\begin{align}
	\Phi^k(\hat{\xb}^k) - \Phi^k(\xb)
	&~+(\tsum_{t=1}^{T_k}\lambda_t)^{-1}\left[\eta_k(1+\beta^{(k)}_{T_k})\lambda_{T_k}\mathbf{V}(\xb^k,\xb) + \tsum_{t=1}^{T_k}\tsum_{i=1}^m\lambda_t\langle\delta_i^{t-1,k}, x_i-u_i^{t-1}\rangle\right] \\\nn
	&~~~\le  (\tsum_{t=1}^{T_k}\lambda_t)^{-1}\left[(\eta_k\beta^{(k)}_1-\mu/\mathcal{C})\lambda_1\mathbf{V}(\xb^{k-1},\xb)
	+ \tsum_{t=1}^{T_k}\tsum_{i=1}^m\tfrac{(M+\|\delta_i^{t-1,k}\|_*)^2\lambda_t}{2 \eta_k\beta_t}\right], \ \forall \xb\in X^m.
	\end{align}
	By plugging into the above relation the values of $\lambda_t$ and $\beta^{(k)}_t$ in \eqref{eqn:csparam_sc}, together with the definition of $\Phi^k$ in \eqref{eqn:defPhi_k} and rearranging the terms, we have
	\begin{align*}
	\la \Lb(\hat \xb^k-\xb), \yb^k\ra + F(\hat \xb^k)- F(\xb)
	&~\le \eta_k\mathbf{V}(\xb^{k-1},\xb) -(\mu/\mathcal{C}+\eta_k)\mathbf{V}(\xb^k,\xb)-\eta_k\mathbf{V}(\xb^{k-1},\hat \xb^k)\\
	&~~~+\tfrac{2}{T_k(T_k+1)}\tsum_{t=1}^{T_k}\tsum_{i=1}^m\left[t\langle\delta_i^{t-1,k}, x_i- u_i^{t-1}\rangle
	+ \tfrac{(M+\|\delta_i^{t-1,k}\|_*)^2t}{(t+1)\mu/\mathcal{C}+(t-1)\eta_k}\right], \, \forall \xb\in X^m, \, k\ge 1.
	\end{align*}
In view of \eqref{dual_opt}, the above relation and the definition of Q in \eqref{eqn:gap}, and following the same trick that we used to obtain \eqref{Q_tilde_D}, we have, for all $\zb\in X^m\times \bbr^{md}$,
	\begin{align}\label{Q_bar_D}
	\tsum_{k=1}^N\theta_kQ(\hat \xb^k,\yb^k;\zb)
	&\le \tsum_{k=1}^N\theta_k\bar \Delta_k
	+ \tsum_{k=1}^N\tsum_{t=1}^{T_k}\tsum_{i=1}^m\tfrac{2\theta_k}{T_k(T_k+1)}\left[t\langle\delta_i^{t-1,k}, x_i-u_i^{t-1}\rangle
	+ \tfrac{(M+\|\delta_i^{t-1,k}\|_*)^2t}{(t+1)\mu/\mathcal{C}+(t-1)\eta_k}\right],
	\end{align}
	where
	\begin{align}\label{def_bDelta}
	\bar \Delta_k:=&\Lb (\hat \xb^k-\tilde \xb^{k}),\yb-\yb^k\ra
	+\eta_k\mathbf{V}(\xb^{k-1},\xb)-(\mu/\mathcal{C}+\eta_k)\mathbf{V}(\xb^k,\xb)-\eta_k\mathbf{V}(\xb^{k-1},\hat \xb^k)\\\nn
	&~~~+\tfrac{\tau_k}{2}\left[\|\yb-\yb^{k-1}\|^2-\|\yb-\yb^k\|^2-\|\yb^{k-1}-\yb^k\|^2\right].
	\end{align}
	Since $\bar \Delta_k$ in \eqref{def_bDelta} shares a similar structure with $\tilde \Delta_k$ in \eqref{def_tDelta} (also $\Delta_k$ in \eqref{def_Delta}),
	we can follow similar procedure as in \eqref{Delta_recursion} to simplify the RHS of \eqref{Q_bar_D}. Note that the only difference of \eqref{def_bDelta} and \eqref{def_tDelta} (also \eqref{def_Delta}) is in the coefficient of the terms $\mathbf{V}(\xb^{k-1},\xb)$, and $\mathbf{V}(\xb^k,\xb)$. Hence, by using condition \eqref{theta_eta_ds} in place of \eqref{theta_eta_d} (also \eqref{theta_eta}) , 
we obtain $\forall \zb\in X^m\times \bbr^{md}$
	\begin{align}\label{Q_bnd_s}
	\tsum_{k=1}^N\theta_kQ(\hat \xb^k,\yb^k;\zb)
	&\le \theta_1\eta_1\mathbf{V}(\xb^0,\xb)+\tfrac{\theta_1\tau_1}{2}\|\yb^0\|^2+ \la \hat\vb,\yb\ra\\\nn
	&~~~+\tsum_{k=1}^N\tsum_{t=1}^{T_k}\tsum_{i=1}^m\tfrac{2\theta_k}{T_k(T_k+1)}\left[t\langle\delta_i^{t-1,k}, x_i-u_i^{t-1}\rangle
	+ \tfrac{2t(M^2+\|\delta_i^{t-1,k}\|_*^2)}{(t+1)\mu/\mathcal{C}+(t-1)\eta_k}\right],
	\end{align}
where $\hat\vb$ is defined in \eqref{eqn:hatvb}.
	Our result in \eqref{Q_expectation2} immediately follows from the convexity of $Q$.

    Following the same procedure as we obtain \eqref{bnd_xNs_ns}, for any saddle point $\zb^*=(\xb^*,\yb^*)$ of \eqref{eqn:saddle}, we have 
    \begin{align}\label{eqn:xNbndstochp}
     \tfrac{\theta_N\eta_N }{2}\|\hat \xb^N-\xb^{N-1}\|^2
    &\le \tfrac{\theta_N\|\Lb\|^2}{2\tau_N}\|\xb^N-\xb^{N-1}\|^2 +\theta_1\eta_1\mathbf{V}(\xb^0, \xb^*)+\tfrac{\theta_1\tau_1}{2}\|\yb^*-\yb^0\|^2\\\nn
    &~~~+\tsum_{k=1}^N\tsum_{t=1}^{T_k}\tsum_{i=1}^m\tfrac{2\theta_k}{T_k(T_k+1)}\left[t\langle\delta_i^{t-1,k}, x_i^*-u_i^{t-1}\rangle
   + \tfrac{2t(M^2+\|\delta_i^{t-1,k}\|_*^2)}{(t+1)\mu/\mathcal{C}+(t-1)\eta_k}\right],\\\nn
   \tfrac{\theta_N\tau_N}{2}\|\yb^*-\yb^N\|^2
   &\le \tfrac{\theta_N\|\Lb\|^2}{2\eta_N }\|\yb^*-\yb^N\|^2+\theta_1\eta_1\mathbf{V}(\xb^0, \xb^*)+\tfrac{\theta_1\tau_1}{2}\|\yb^*-\yb^0\|^2\\\nn
   &~~~+\tsum_{k=1}^N\tsum_{t=1}^{T_k}\tsum_{i=1}^m\tfrac{2\theta_k}{T_k(T_k+1)}\left[t\langle\delta_i^{t-1,k}, x_i^*-u_i^{t-1}\rangle
  + \tfrac{2t(M^2+\|\delta_i^{t-1,k}\|_*^2)}{(t+1)\mu/\mathcal{C}+(t-1)\eta_k}\right],
    \end{align}
    from which the desired result in \eqref{eqn:xNbndstoch} follows.\\

We are ready to provide proofs for Lemma~\ref{lem:main_deter} and \ref{lem:main_deter_sc}, which demonstrates the convergence properties of the deterministic communication sliding method.
\vgap

\noindent{\bf Proof of Lemma~\ref{lem:main_deter}}\\
When $f_i, \ i=1,\dots,m$ are general nonsmooth convex functions, we have $\delta_i^t=0$, $\mu=0$ and $M>0$. Therefore, in view of \eqref{Q_bnd}, we have, $\forall \zb \in X^m\times \bbr^{md}$,
	\begin{align*}
	\tsum_{k=1}^N\theta_kQ(\hat \xb^k,\yb^k;\zb)
	&\le \tfrac{(T_1+1)(T_1+2)\theta_1\eta_1}{T_1(T_1+3)}\mathbf{V}(\xb^0,\xb)+\tfrac{\theta_1\tau_1}{2}\|\yb^0\|^2 +\la\hat\vb,\yb\ra
	+ \tsum_{k=1}^N\tfrac{4mM^2\theta_k}{(T_k+3)\eta_k},
	\end{align*}
where $\hat\vb$ is defined in \eqref{eqn:hatvb}.
	Our result in \eqref{eqn:Q0} immediately follows from the convexity of $Q$.
	Moreover, our result in \eqref{xNs_bnd_ns} follows from setting $\delta_i^{t-1,k}=0$ in \eqref{bnd_xNs_ns} and \eqref{bnd_yNs_ns}.

\vgap
\noindent{\bf Proof of Lemma~\ref{lem:main_deter_sc}}\\
When $f_i, \ i=1,\dots,m$ are strongly convex functions, we have $\delta_i^t=0$ and $\mu,\ M>0$. Therefore, in view of \eqref{Q_bnd_s}, we obtain, $\forall \zb\in X^m\times \bbr^{md}$,
    \begin{align*}
    \tsum_{k=1}^N\theta_kQ(\hat \xb^k,\yb^k;\zb)
    &\le \theta_1\eta_1\mathbf{V}(\xb^0,\xb)+\tfrac{\theta_1\tau_1}{2}\|\yb^0\|^2+\la\hat\vb,\yb\ra
    +\tsum_{k=1}^N\tsum_{t=1}^{T_k}\tfrac{2mM^2\theta_k}{T_k(T_k+1)}\tfrac{t} {(t+1)\mu/\mathcal{C}+(t-1)\eta_k},
    \end{align*}
    where $\hat\vb$ is defined in \eqref{eqn:hatvb}.
	Our result in \eqref{eqn:Q02} immediately follows from the convexity of $Q$.
Also, the result in \eqref{eqn:xNbnd} follows by setting $\delta_i^{t-1,k} = 0$ in \eqref{eqn:xNbndstochp}.\\

\noindent{\bf Proof of Theorem~\ref{thm:main_stoch_prob}}\\
Observe that by Assumption~\eqref{assume:unbiased}, \eqref{assume:sm_bounded} and~\eqref{assume:light_tail} on the SO and the definition of $u_i^{t,k}$, the sequence $\{\la \delta_i^{t-1,k}, x_i^*-u_i^{t-1,k}\ra\}_{1\le i\le m, 1\le t\le T_k, k\ge 1}$ is a martingale-difference sequence. Denoting
	\[
	\gamma_{k,t}:=\tfrac{\theta_k\lambda_t}{\tsum_{t=1}^{T_k}\lambda_t},
	\]
	and using the large-deviation theorem for martingale-difference sequence (e.g. Lemma 2 of \cite{lns11}) and the fact that
	\begin{align*}
	&\bbe[\exp\{\gamma_{k,t}^2\la \delta_i^{t-1,k}, x_i^*-u_i^{t-1,k}\ra^2/(2\gamma_{k,t}^2\bar V_i(x_i^*)\sigma^2) \}]\\
	&~~~\le \bbe[\exp\{\| \delta_i^{t-1,k}\|_*^2, \|x_i^*-u_i^{t-1,k}\|^2/(2\bar V_i(x_i^*)\sigma^2) \}]\\
	&~~~\le \bbe[\exp\{\| \delta_i^{t-1,k}\|_*^2/\sigma^2\}]\le \exp\{1\},
	\end{align*}
	we conclude that, $\forall \zeta>0$,
	\begin{align}\label{prob_bnd1}
	\Prob&\left\{\tsum_{k=1}^N\tsum_{t=1}^{T_k}\tsum_{i=1}^m\gamma_{k,t}\la \delta_i^{t-1,k}, u_i^{t-1,k}-x_i^*\ra>\zeta \sigma \sqrt{2\bar {\mathbf{V}}(\xb^*)\tsum_{k=1}^N\tsum_{t=1}^{T_k}\gamma_{k,t}^2}\right\}
	\le \exp\{-\zeta^2/3\}.
	\end{align}
	Now let
	\[
	S_{k,t}:=\tfrac{\theta_k\lambda_t}{\left(\tsum_{t=1}^{T_k}\lambda_t\right)\eta_k\beta_t},
	\]
	and $S:=\tsum_{k=1}^N\tsum_{t=1}^{T_k}\tsum_{i=1}^mS_{k,t}$. By the convexity of exponential function, we have
	\begin{align*}
	&\bbe[\exp\{\tfrac{1}{S}\tsum_{k=1}^N\tsum_{t=1}^{T_k}\tsum_{i=1}^mS_{k,t}\|\delta_i^{t-1,k}\|_*^2/\sigma^2\}]
	\le \bbe[\tfrac{1}{S}\tsum_{k=1}^N\tsum_{t=1}^{T_k}\tsum_{i=1}^mS_{k,t}\exp\{\|\delta_i^{t-1,k}\|_*^2/\sigma^2\}]\le \exp\{1\},
	\end{align*}
	where the last inequality follows from Assumption~\eqref{assume:light_tail}. Therefore, by Markov's inequality, for all $\zeta>0$,
	\begin{align}\label{prob_bnd2}
	&\Prob\left\{\tsum_{k=1}^N\tsum_{t=1}^{T_k}\tsum_{i=1}^mS_{k,t}\|\delta_i^{t-1,k}\|_*^2>(1+\zeta)\sigma^2\tsum_{k=1}^N\tsum_{t=1}^{T_k}\tsum_{i=1}^mS_{k,t}\right\}\\\nn
	&~~~=\Prob\left\{\exp \left\{\tfrac{1}{S}\tsum_{k=1}^N\tsum_{t=1}^{T_k}\tsum_{i=1}^mS_{k,t}\|\delta_i^{t-1,k}\|_*^2/\sigma^2\right\}\ge \exp\{1+\zeta\}\right\}\le \exp\{-\zeta\}.
	\end{align}
	Combing \eqref{prob_bnd1}, \eqref{prob_bnd2}, \eqref{Q_pert_bnd} and \eqref{eqn:gapp}, our result in \eqref{gap_prob} immediately follows.

\setcounter{equation}{0}

\section{Concluding Remarks}
\label{sec:con}
In this paper, we present a new class of decentralized primal-dual methods which can significantly
reduce the number of inter-node communications required to solve the
distributed optimization problem in \eqref{eqn:orgprob}.
More specifically, we show that by using these algorithms, the
total number of communication rounds can be significantly reduced
to ${\cal O}(1/\epsilon)$ when the objective functions $f_i$'s are convex and not necessarily smooth.
By properly designing the communication sliding algorithms, we demonstrate
that the ${\cal O}(1/\epsilon)$  number of communications can still be maintained for general convex objective functions (and it can be further reduced to ${\cal O}(1/\sqrt{\epsilon})$ for strongly convex objective functions)
even if the local subproblems are solved inexactly through iterative procedure (cf. \SGD~procedure) by the network agents.
In this case, the number of intra-node subgradient computations that we need will be bounded by ${\cal O}(1/\epsilon^2)$ (resp., ${\cal O}(1/\epsilon)$) when the objective functions $f_i$'s are convex (resp., strongly convex), which is comparable to that required in centralized nonsmooth optimization and not improvable in general.
We also establish similar complexity bounds for solving stochastic decentralized optimization counterpart by developing the stochastic communication sliding methods,
which can provide
communication-efficient ways to deal with streaming data and decentralized statistical inference.
All these decentralized communication sliding algorithms have the potential to significantly increase the performance of multiagent systems,
where the bottleneck exists in the communication.




\begin{thebibliography}{10}

\bibitem{arrow1958studies}
K.~Arrow, L.~Hurwicz, and H.~Uzawa.
\newblock {\em Studies in Linear and Non-linear Programming}.
\newblock Stanford Mathematical Studies in the Social Sciences. Stanford
  University Press, 1958.

\bibitem{Bertsekas-IP}
D.~P. Bertsekas.
\newblock Incremental proximal methods for large scale convex optimization.
\newblock {\em Mathematical Programming}, 129:163–--195, 2011.

\bibitem{Bertsekas-aggregated}
D.~P. Bertsekas.
\newblock Incremental aggregated proximal and augmented lagrangian algorithms.
\newblock Technical Report LIDS-P-3176, Laboratory for Information and Decision
  Systems, 2015.

\bibitem{Boyd-ADMM}
S.~Boyd, N.~Parikh, E.~Chu, B.~Peleato, and J.~Eckstein.
\newblock Distributed optimization and statistical learning via the alternating
  direction method of multipliers.
\newblock {\em Found. Trends Mach. Learn.}, 3(1):1--122, January 2011.

\bibitem{BREGMAN1967}
L.M. Bregman.
\newblock The relaxation method of finding the common point of convex sets and
  its application to the solution of problems in convex programming.
\newblock {\em USSR Computational Mathematics and Mathematical Physics},
  7(3):200 -- 217, 1967.

\bibitem{ChamPoc14-1}
A.~Chambolle and T.~Pock.
\newblock On the ergodic convergence rates of a first-order primal-dual
  algorithm.
\newblock Oct. 30, 2014.

\bibitem{Chambolle-PD}
Antonin Chambolle and Thomas Pock.
\newblock A first-order primal-dual algorithm for convex problems with
  applications to imaging.
\newblock {\em J. Math. Imaging Vis.}, 40(1):120--145, May 2011.

\bibitem{Chang-Stochastic}
T.~Chang and M.~Hong.
\newblock Stochastic proximal gradient consensus over random networks.
\newblock http://arxiv.org/abs/1511.08905, 2015.

\bibitem{InexactConsensus}
T.~Chang, M.~Hong, and X.~Wang.
\newblock Multi-agent distributed optimization via inexact consensus admm.
\newblock http://arxiv.org/abs/1402.6065, 2014.

\bibitem{perturbedPD}
T.-H. Chang, A.~Nedi\'c, and A.~Scaglione.
\newblock Distributed constrained optimization by consensus-based primal-dual
  perturbation method.
\newblock {\em Automatic Control, IEEE Transactions on}, 59(6):1524--1538, June
  2014.

\bibitem{Chen-fastProx}
A.~Chen and A.~Ozdaglar.
\newblock A fast distributed proximal gradient method.
\newblock In {\em Communication, Control, and Computing (Allerton), 2012 50th
  Annual Allerton Conference on}, pages 601--608, Oct 2012.

\bibitem{CheLanOu13-1}
Y.~Chen, G.~Lan, and Y.~Ouyang.
\newblock Optimal primal-dual methods for a class of saddle point problems.
\newblock 24(4):1779--1814, 2014.

\bibitem{Dang-Lan}
C.~Dang and G.~Lan.
\newblock Randomized first-order methods for saddle point optimization.
\newblock Technical Report 32611, Department of Industrial and Systems
  Engineering, University of Florida, Gainesville, FL, 2015.

\bibitem{Duchi-DDA}
J.~Duchi, A.~Agarwal, and M.~Wainwright.
\newblock Dual averaging for distributed optimization: Convergence analysis and
  network scaling.
\newblock {\em {IEEE} Trans. Automat. Contr.}, 57(3):592--606, 2012.

\bibitem{Durham-Bullo}
J.~W. Durham, A.~Franchi, and F.~Bullo.
\newblock Distributed pursuit-evasion without mapping or global localization
  via local frontiers.
\newblock {\em Autonomous Robots}, 32(1):81--95, 2012.

\bibitem{GhaLan10-1}
S.~Ghadimi and G.~Lan.
\newblock Optimal stochastic approximation algorithms for strongly convex
  stochastic composite optimization {I:} {A} generic algorithmic framework.
\newblock {\em {SIAM} Journal on Optimization}, 22(4):1469--1492, 2012.

\bibitem{ghadimi2013optimal}
S.~Ghadimi and G.~Lan.
\newblock Optimal stochastic approximation algorithms for strongly convex
  stochastic composite optimization, ii: shrinking procedures and optimal
  algorithms.
\newblock {\em SIAM Journal on Optimization}, 23(4):2061--2089, 2013.

\bibitem{Parrilo-aggregated}
M.~Gurbuzbalaban, A.~Ozdaglar, and P.~Parrilo.
\newblock On the convergence rate of incremental aggregated gradient
  algorithms.
\newblock http://arxiv.org/abs/1506.02081, 2015.

\bibitem{he2012on}
B.~He and X.~Yuan.
\newblock On the o(1/n) convergence rate of the douglas-rachford alternating
  direction method.
\newblock {\em SIAM Journal on Numerical Analysis}, 50(2):700--709, 2012.

\bibitem{HeJudNem15-1}
N.~He, A.~Juditsky, and A.~Nemirovski.
\newblock Mirror prox algorithm for multi-term composite minimization and
  semi-separable problems.
\newblock {\em Journal of Computational Optimization and Applications},
  103:127--152, 2015.

\bibitem{con01}
A.~Jadbabaie, Jie Lin, and A.S. Morse.
\newblock Coordination of groups of mobile autonomous agents using nearest
  neighbor rules.
\newblock {\em IEEE Transactions on Automatic Control}, 48(6):988 -- 1001, June
  2003.

\bibitem{Jakovetic-Fast}
D.~Jakovetic, J.~Xavier, and J.~Moura.
\newblock Fast distributed gradient methods.
\newblock {\em Automatic Control, IEEE Transactions on}, 59(5):1131--1145, May
  2014.

\bibitem{Lan10-3}
G.~Lan.
\newblock An optimal method for stochastic composite optimization.
\newblock {\em Mathematical Programming}, 133(1):365--397, 2012.

\bibitem{Lan-Sliding}
G.~Lan.
\newblock Gradient sliding for composite optimization.
\newblock {\em Mathematical Programming}, 159(1):201--235, 2016.

\bibitem{lns11}
G.~Lan, A.~Nemirovski, and A.~Shapiro.
\newblock Validation analysis of mirror descent stochastic approximation
  method.
\newblock {\em Math. Program.}, 134(2):425--458, 2012.

\bibitem{GLYZ}
G.~Lan and Y.~Zhou.
\newblock An optimal randomized incremental gradient method.
\newblock http://arxiv.org/abs/1507.02000, 2015.

\bibitem{Lobel2011}
I.~Lobel and A.~Ozdaglar.
\newblock Distributed subgradient methods for convex optimization over random
  networks.
\newblock {\em IEEE Transactions on Automatic Control}, 56(6):1291 --1306, June
  2011.

\bibitem{ADMM-linear}
A.~Makhdoumi and A.~Ozdaglar.
\newblock Convergence rate of distributed admm over networks.
\newblock http://arxiv.org/abs/1601.00194, 2016.

\bibitem{Ribeiro-DQM}
A.~Mokhtari, W.~Shi, Q.~Ling, and A.~Ribeiro.
\newblock Dqm: Decentralized quadratically approximated alternating direction
  method of multipliers.
\newblock http://arxiv.org/abs/1508.02073, 2015.

\bibitem{Ribeiro-second}
A.~Mokhtari, W.~Shi, Q.~Ling, and A.~Ribeiro.
\newblock A decentralized second-order method with exact linear convergence
  rate for consensus optimization.
\newblock http://arxiv.org/abs/1602.00596, 2016.

\bibitem{Monteiro01}
R.~D.~C. Monteiro and B.~F. Svaiter.
\newblock On the complexity of the hybrid proximal extragradient method for the
  iterates and the ergodic mean.
\newblock {\em SIAM Journal on Optimization}, 20(6):2755--2787, 2010.

\bibitem{Monteiro02}
R.~D.~C. Monteiro and B.~F. Svaiter.
\newblock Complexity of variants of tseng's modified f-b splitting and
  korpelevich's methods for hemivariational inequalities with applications to
  saddle-point and convex optimization problems.
\newblock {\em SIAM Journal on Optimization}, 21(4):1688--1720, 2011.

\bibitem{Monteiro03}
R.~D.~C. Monteiro and B.~F. Svaiter.
\newblock Iteration-complexity of block-decomposition algorithms and the
  alternating direction method of multipliers.
\newblock {\em SIAM Journal on Optimization}, 23(1):475--507, 2013.

\bibitem{MonSva10-1}
R.D.C. Monteiro and B.F. Svaiter.
\newblock On the complexity of the hybrid proximal projection method for the
  iterates and the ergodic mean.
\newblock 20:2755--2787, 2010.

\bibitem{Nedic11}
A.~Nedi\'c.
\newblock Asynchronous broadcast-based convex optimization over a network.
\newblock {\em IEEE Trans. Automat. Contr.}, 56(6):1337--1351, 2011.

\bibitem{AN2001}
A.~Nedi\'c, D.~P. Bertsekas, and V.~S. Borkar.
\newblock Distributed asynchronous incremental subgradient methods.
\newblock {\em Inherently Parallel Algorithms in Feasibility and Optimization
  and Their Applications}, pages 311--407, 2001.

\bibitem{ANAO}
A.~Nedi\'c and A.~Olshevsky.
\newblock Distributed optimization over time-varying directed graphs.
\newblock {\em IEEE Transactions on Automatic Control}, 60(3):601--615, March
  2015.

\bibitem{Wilbur-TV}
A.~Nedi\'c, A.~Olshevsky, and W.~Shi.
\newblock Achieving geometric convergence for distributed optimization over
  time-varying graphs.
\newblock http://arxiv.org/abs/1607.03218, 2016.

\bibitem{Nedic2009}
A.~Nedi\'c and A.~Ozdaglar.
\newblock Distributed subgradient methods for multi-agent optimization.
\newblock {\em IEEE Transactions on Automatic Control}, 54(1):48--61, 2009.

\bibitem{Nem05-1}
A.~S. Nemirovski.
\newblock Prox-method with rate of convergence $o(1/t)$ for variational
  inequalities with lipschitz continuous monotone operators and smooth
  convex-concave saddle point problems.
\newblock 15:229--251, 2005.

\bibitem{NJLS09-1}
A.~S. Nemirovski, A.~Juditsky, G.~Lan, and A.~Shapiro.
\newblock Robust stochastic approximation approach to stochastic programming.
\newblock 19:1574--1609, 2009.

\bibitem{nemyud:83}
A.~S. Nemirovski and D.~Yudin.
\newblock {\em Problem complexity and method efficiency in optimization}.
\newblock Wiley-Interscience Series in Discrete Mathematics. John Wiley, XV,
  1983.

\bibitem{Nest05-1}
Y.~E. Nesterov.
\newblock Smooth minimization of nonsmooth functions.
\newblock {\em Mathematical Programming}, 61(2):275--319, 2015.

\bibitem{GL-AADMM}
Y.~Ouyang, Y.~Chen, G.~Lan, and E.~Pasiliao Jr.
\newblock An accelerated linearized alternating direction method of
  multipliers.
\newblock {\em SIAM Journal on Imaging Sciences}, 8(1):644--681, 2015.

\bibitem{NaLi-Harness}
G.~Qu and N.~Li.
\newblock Harnessing smoothness to accelerate distributed optimization.
\newblock http://arxiv.org/abs/1605.07112, 2016.

\bibitem{Rabbat-SMD}
M.~Rabbat.
\newblock Multi-agent mirror descent for decentralized stochastic optimization.
\newblock In {\em 2015 IEEE 6th International Workshop on Computational
  Advances in Multi-Sensor Adaptive Processing (CAMSAP)}, pages 517--520, Dec
  2015.

\bibitem{rabbat}
M.~Rabbat and R.~D. Nowak.
\newblock Distributed optimization in sensor networks.
\newblock In {\em IPSN}, pages 20--27, 2004.

\bibitem{Ram:2009}
S.~S. Ram, A.~Nedi\'{c}, and V.~V. Veeravalli.
\newblock Incremental stochastic subgradient algorithms for convex
  optimization.
\newblock {\em SIAM J. on Optimization}, 20(2):691--717, June 2009.

\bibitem{Ram2010}
S.~S. Ram, A.~Nedi\'c, and V.~V. Veeravalli.
\newblock {Distributed Stochastic Subgradient Projection Algorithms for Convex
  Optimization}.
\newblock {\em Journal of Optimization Theory and Applications}, 147:516--545,
  2010.

\bibitem{ram_info}
S.~S. Ram, V.~V. Veeravalli, and A.~Nedi\'c.
\newblock Distributed non-autonomous power control through distributed convex
  optimization.
\newblock In {\em IEEE INFOCOM}, pages 3001--3005, 2009.

\bibitem{Wilbur-ADMM}
W.~Shi, Q.~Ling, G.~Wu, and W.~Yin.
\newblock On the linear convergence of the admm in decentralized consensus
  optimization.
\newblock {\em IEEE Transactions on Signal Processing}, 62(7):1750--1761, 2014.

\bibitem{WYin-Extra}
W.~Shi, Q.~Ling, G.~Wu, and W.~Yin.
\newblock Extra: An exact first-order algorithm for decentralized consensus
  optimization.
\newblock {\em SIAM Journal on Optimization}, 25(2):944–--966, 2015.

\bibitem{WYin-PGExtra}
W.~Shi, Q.~Ling, G.~Wu, and W.~Yin.
\newblock A proximal gradient algorithm for decentralized composite
  optimization.
\newblock {\em IEEE Transactions on Signal Processing}, 63(22):6013--–6023,
  November 2015.

\bibitem{DistStoch}
A.~Simonetto, L.~Kester, and G.~Leus.
\newblock Distributed time-varying stochastic optimization and utility-based
  communication.
\newblock http://arxiv.org/abs/1408.5294, 2014.

\bibitem{dual-decompos}
H.~Terelius, U.~Topcu, and R.~Murray.
\newblock Decentralized multi-agent optimization via dual decomposition.
\newblock {\em IFAC Proceedings Volumes}, 44(1):11245--11251, 2011.

\bibitem{tsianos2012consensus}
K.~Tsianos, S.~Lawlor, and M.~Rabbat.
\newblock Consensus-based distributed optimization: Practical issues and
  applications in large-scale machine learning.
\newblock In {\em Proceedings of the 50th Allerton Conference on Communication,
  Control, and Computing}, 2012.

\bibitem{tsianos-pushsum}
K.~Tsianos, S.~Lawlor, and M.~Rabbat.
\newblock Push-sum distributed dual-averaging for convex optimization.
\newblock In {\em Proceedings of the 51st IEEE Conference on Decision and
  Control}, pages 5453--5458, Maui, Hawaii, December 2012.

\bibitem{Rabbat-online}
K.~Tsianos and M.~Rabbat.
\newblock Consensus-based distributed online prediction and optimization.
\newblock In {\em 2013 IEEE Global Conference on Signal and Information
  Processing}, pages 807--810, Dec 2013.

\bibitem{Tsi1986}
J.~Tsitsiklis, D.~Bertsekas, and M.~Athans.
\newblock Distributed asynchronous deterministic and stochastic gradient
  optimization algorithms.
\newblock {\em IEEE Transactions on Automatic Control}, 31(9):803 -- 812, Sep.
  1986.

\bibitem{Tsiphd}
J.~N. Tsitsiklis.
\newblock {\em Problems in Decentralized Decision Making and Computation}.
\newblock PhD thesis, Massachusetts Inst. Technol., Cambridge, MA, 1984.

\bibitem{Wang-Bertsekas}
M.~Wang and D.~P. Bertsekas.
\newblock Incremental constraint projection-proximal methods for nonsmooth
  convex optimization.
\newblock Technical Report LIDS-P-2907, Laboratory for Information and Decision
  Systems, 2013.

\bibitem{Wei-admm}
E.~Wei and A.~Ozdaglar.
\newblock On the ${O}(1/k)$ convergence of asynchronous distributed alternating
  direction method of multipliers.
\newblock http://arxiv.org/pdf/1307.8254, 2013.

\bibitem{Khan-MD}
C.~Xi, Q.~Wu, and U.~A. Khan.
\newblock Distributed mirror descent over directed graphs.
\newblock http://arxiv.org/abs/1412.5526, 2014.

\bibitem{Yuchen14}
Y.~Zhang and L.~Xiao.
\newblock Stochastic primal-dual coordinate method for regularized empirical
  risk minimization.
\newblock In {\em Proceedings of the 32nd International Conference on Machine
  Learning}, pages 353--361, 2015.

\bibitem{Martinez-PD}
M.~Zhu and S.~Martinez.
\newblock On distributed convex optimization under inequality and equality
  constraints.
\newblock {\em Automatic Control, IEEE Transactions on}, 57(1):151--164, Jan
  2012.

\end{thebibliography}

\end{document}